\documentclass[notitlepage,11pt,reqno]{amsart}
\usepackage[foot]{amsaddr}

\RequirePackage[normalem]{ulem}
\RequirePackage{amssymb,nicefrac,bm,upgreek,mathtools,verbatim}
\RequirePackage[mathscr]{eucal}
\RequirePackage{dsfont}
\RequirePackage{accents}
\RequirePackage[final]{hyperref}
\usepackage{csquotes}

\usepackage[margin=1in]{geometry}
\allowdisplaybreaks

\newcommand{\stkout}[1]{\ifmmode\text{\sout{\ensuremath{#1}}}\else\sout{#1}\fi}
\usepackage{color}
\definecolor{dmagenta}{rgb}{.6,.1,.5}
\definecolor{dblue}{rgb}{.0,.0,.5}
\definecolor{mblue}{rgb}{.0,.0,.7}
\definecolor{ddblue}{rgb}{.0,.0,.4}
\definecolor{dred}{rgb}{.6,.0,.0}
\definecolor{dgreen}{rgb}{.0,.5,.0}
\definecolor{Eeom}{rgb}{.0,.0,.5}

\usepackage[normalem]{ulem}
\newtheorem{lemma}{Lemma}[section]
\newtheorem{theorem}{Theorem}[section]
\newtheorem{proposition}{Proposition}[section]
\newtheorem{corollary}{Corollary}[section]

\theoremstyle{definition}
\newtheorem{definition}{Definition}[section]
\newtheorem{assumption}{Assumption}[section]
\newtheorem{hypothesis}{Hypothesis}[section]

\newtheorem{example}{Example}[section]

\theoremstyle{remark}
\newtheorem{remark}{Remark}[section]
\numberwithin{equation}{section}
\hypersetup{
  colorlinks=true,
  citecolor=mblue,
  linkcolor=mblue,
  frenchlinks=false,
  pdfborder={0 0 0},
  naturalnames=false,
  hypertexnames=false,
  breaklinks}
\usepackage[capitalize,nameinlink]{cleveref}
\usepackage[abbrev,msc-links,nobysame]{amsrefs} 
\crefname{section}{Section}{Sections}
\crefname{subsection}{Section}{Sections}
\crefname{hypothesis}{Hypothesis}{Conditions}
\crefname{assumption}{Assumption}{Assumptions}
\crefname{lemma}{Lemma}{Lemmas}
\Crefname{figure}{Figure}{Figures}

\crefformat{equation}{\textup{#2(#1)#3}}
\crefrangeformat{equation}{\textup{#3(#1)#4--#5(#2)#6}}
\crefmultiformat{equation}{\textup{#2(#1)#3}}{ and \textup{#2(#1)#3}}
{, \textup{#2(#1)#3}}{, and \textup{#2(#1)#3}}
\crefrangemultiformat{equation}{\textup{#3(#1)#4--#5(#2)#6}}%
{ and \textup{#3(#1)#4--#5(#2)#6}}{, \textup{#3(#1)#4--#5(#2)#6}}%
{, and \textup{#3(#1)#4--#5(#2)#6}}

\Crefformat{equation}{#2Equation~\textup{(#1)}#3}
\Crefrangeformat{equation}{Equations~\textup{#3(#1)#4--#5(#2)#6}}
\Crefmultiformat{equation}{Equations~\textup{#2(#1)#3}}{ and \textup{#2(#1)#3}}
{, \textup{#2(#1)#3}}{, and \textup{#2(#1)#3}}
\Crefrangemultiformat{equation}{Equations~\textup{#3(#1)#4--#5(#2)#6}}%
{ and \textup{#3(#1)#4--#5(#2)#6}}{, \textup{#3(#1)#4--#5(#2)#6}}%
{, and \textup{#3(#1)#4--#5(#2)#6}}

\crefdefaultlabelformat{#2\textup{#1}#3}
\newcommand{\ttup}[1]{\textup{(}#1\textup{)}}

\newcommand{\cA}{{\mathcal{A}}}   
\newcommand{\sA}{{\mathscr{A}}}   
\newcommand{\sB}{{\mathscr{B}}}
\newcommand{\DD}{\mathbb{D}}
\newcommand{\cD}{{\mathcal{D}}}   
\newcommand{\cG}{{\mathcal{G}}}   %
\newcommand{\cI}{{\mathcal{I}}}   %
\newcommand{\cJ}{{\mathcal{J}}}   %
\newcommand{\cK}{{\mathcal{K}}}   
\newcommand{\cL}{{\mathcal{L}}}   
\newcommand{\cP}{{\mathcal{P}}}   
\newcommand{\cQ}{{\mathcal{Q}}}   
\newcommand{\cR}{{\mathcal{R}}}   
\newcommand{\Lyap}{{\mathcal{V}}} 
\newcommand{\sX}{{\mathscr{X}}}   
\newcommand{\sZ}{{\mathscr{Z}}}   
\newcommand{\sV}{\mathscr{V}}     
\newcommand{\fX}{{\mathfrak{X}}}  
\newcommand{\cT}{{\mathcal{T}}}   

\newcommand{\RR}{\mathds{R}}
\newcommand{\NN}{\mathds{N}}
\newcommand{\ZZ}{\mathds{Z}}

\newcommand{\Rd}{\mathds{R}^{d}}
\DeclareMathOperator{\Exp}{\mathbb{E}}

\newcommand{\D}{\mathrm{d}}
\newcommand{\E}{\mathrm{e}}

\newcommand{\Ind}{\mathds{1}}   

\newcommand{\abs}[1]{\lvert#1\rvert}
\newcommand{\norm}[1]{\lVert#1\rVert}
\newcommand{\babs}[1]{\bigl\lvert#1\bigr\rvert}
\newcommand{\Babs}[1]{\Bigl\lvert#1\Bigr\rvert}

\newcommand\transp{^{\mathsf{T}}}
\newcommand{\df}{\coloneqq}

\DeclareMathOperator*{\diag}{diag}

\newcommand{\order}{{\mathscr{O}}}

\newcommand{\grad}{\nabla}

\newcommand{\ttl}{\Large Exponential ergodicity and steady-state approximations for
\\[3pt] a class of Markov processes
under fast regime switching}

\begin{document}
\title[Ergodicity of a class of Markov processes under fast regime switching]{\ttl}

\author[Ari Arapostathis]
{Ari Arapostathis$^\dag$
}
\address{$^\dag$ Department of Electrical and Computer Engineering\\
The University of Texas at Austin\\
2501 Speedway, EERC 7.824\\
Austin, TX~~78712}
\email{ari@utexas.edu}

\author[Guodong Pang]
{Guodong Pang$^\ddag$
}
\author[Yi Zheng]{Yi Zheng$^\ddag$
}
\address{$^\ddag$ The Harold and Inge Marcus Department of Industrial and
Manufacturing Engineering,
College of Engineering,
Pennsylvania State University,
University Park, PA 16802}
\email{$\lbrace$gup3,yxz282$\rbrace$@psu.edu}

\begin{abstract} 
We study ergodic properties
of a class of Markov-modulated general birth-death processes under fast regime switching.
The first set of results concerns the ergodic properties of 
the properly scaled joint Markov process with a parameter that is taken large.
Under very weak hypotheses, we show that if
the averaged process is exponentially ergodic for large
values of the parameter, then the same applies to the original  joint Markov process.
The second set of results concerns steady-state diffusion approximations,
under the assumption that the `averaged' fluid limit exists.
Here, we establish convergence rates for the moments of the approximating diffusion process to 
those of the Markov modulated birth-death process.
This is accomplished by comparing 
the generator of the  approximating diffusion and that of the joint Markov process.
We also provide several examples which demonstrate how the theory can be applied.
\end{abstract}



\maketitle

\section{Introduction}

There has been a considerable amount of
research on Markov-modulated birth-death processes.
The rate control problem for Markov-modulated single server queue has been addressed
in \cite{RMH13,LQA17,FA09}, while the scheduling control problem for Markov-modulated
critically loaded multiclass many-server queues has been considered in \cite{ADPZ19},
in which exponential ergodicity under a static priority rule is also studied.
The papers \cite{ABMT,JMTW17} address functional limit theorems 
for Markov-modulated Markovian infinite server queues. 
See also the work on the functional limit theorem for Markov-modulated
compound Poisson processes in \cite{PZ17}.
We refer the readers to \cite{Zeifman98,MS14} for the study of stability and 
instability for birth-death processes.

In this paper, we study a class of general birth-death processes
with countable state space and bounded jumps.
Meanwhile, the transition rate functions of the birth-death process 
depend on an underlying continuous time Markov process with finite state space.
An asymptotic framework is considered under which
the Markov-modulated birth-death process is indexed by a scaling parameter $n$, 
with $n$ getting large.
The transition rate matrix of the underlying Markov process 
is of order $n^{\alpha}$, $\alpha>0$, and
the jump size of the birth-death process shrinks at a rate of $n^{\beta}$ with
$\beta\df \max\{\nicefrac{1}{2},1 -\nicefrac{\alpha}{2}\}$.
This scaling has been used in \cite{ABMT,JMTW17,ADPZ19} 
for some special birth-death queueing processes.

In this asymptotic framework, we first provide a sufficient condition 
for the scaled Markov-modulated process
to be exponentially ergodic.
We show that if the `averaged' birth-death process satisfies a Foster-Lyapunov
criterion for a certain
class of Lyapunov functions, then the original Markov-modulated
process also has the same property.
Next, we study steady-state approximations of the Markov-modulated process. 
We construct diffusion models, and show that their steady-state moments approximate
those of the joint Markov process with a rate
$n^{- (\nicefrac{1}{2}\wedge \nicefrac{\alpha}{2})}$.
This problem is motivated by \cite{Gurvich14}, 
in which steady-state approximations for a general birth-death process
have been considered. 
However, the problem is quite challenging in this paper, 
since we need to consider the variabilities of the underlying Markov process,
and the martingale argument in the above referenced work cannot be applied.
We also present some examples from queueing systems
and show that the assumptions presented are easy to verify.

The aforementioned result of exponential ergodicity is stated in \cref{T2.1}.
We consider a large class of scaled Markov-modulated  general birth-death processes,
whose transition rate functions have a linear growth around some distinguished 
point. The state processes are also centered at this point. 
The increments of the transition rate functions are assumed to have affine growth.
This assumption is relaxed in \cref{C2.1}, 
in which a stronger Foster-Lyapunov criterion is required instead. 
The technique of proofs for this set of results is inspired by \cite{Kha12}, which studies
stochastic differential equations with rapid Markovian switching. 
We construct a sequence Lyapunov functions via Poisson equations associated with the 
extended generator of the background Markov process.
The technique employed for our results is more involved, 
since a class of Markov processes under weak hypotheses is considered,
and the scaling parameter affects the state and background processes at the same time. 
In the study of ergodicity of 
a Markov-modulated multiclass $M/M/n + M$ queue under a static priory scheduling policy 
in \cite[Theorem 4]{ADPZ19}, 
the authors observe an effect of `averaged' Halfin-Whitt regime, and also use
a technique similar in spirit to the method in \cite{Kha12}.
In this paper, we consider a more general model which includes 
the one in \cite[Theorem 4]{ADPZ19} as a special case.
In \cref{Ex3.3}, we also present that the result in \cite[Theorem 4]{ADPZ19} 
holds under some weaker condition, and its proof may be simplified a lot following 
the approach as in \cref{C2.1}. 
In \cref{R2.3}, we emphasize that the result in this part can be applied 
in the study of uniformly exponential ergodicity of Markov-modulated multiclass
$M/M/n$ queues  with positive safety staffing.

The main result on steady-state approximations is stated in \cref{T2.2}.
Here, we first construct `averaged' diffusion models, which 
capture the variabilities of the state process and the 
underlying Markov process at the same time.  
In these diffusion models, the variabilities of the state process are asymptotically
negligible at a rate $n^{1 - 2\beta}$ when $\alpha < 1$, while 
the variabilities of the underlying process are asymptotically
negligible at a rate $n^{1 - \alpha}$ when $\alpha > 1$.
The gap between the moments of the steady state of the approximating diffusion models
and those of the joint Markov process shrinks 
at rate of $n^{\nicefrac{\alpha}{2}\wedge \nicefrac{1}{2}}$. 

The result in \cref{T2.2}  extends the results of \cite{Gurvich14} 
to Markov-modulated birth-death processes.
The proofs in \cite{Gurvich14} rely on the gradient estimates of solutions 
of a sequence of Poisson equations associated with diffusions and 
a martingale argument.
Under a uniformly exponential ergodicity assumption for the diffusion models,
the gradient estimates we used for the Poisson equation are the same as those 
found in \cite{Gurvich14}.
However, the martingale argument is difficult to apply in obtaining \cref{T2.2}.
On the other hand, the proof of \cite[Lemma 8]{ADPZ19}
concerning the convergence of mean 
empirical measures for Markov-modulated multiclass $M/M/n + M$
queues  uses a martingale argument, but
considers only compactly supported smooth functions. 
The analogous argument cannot be used in this paper, since we need to consider 
a class of general birth-death processes and the
Lyapunov functions are unbounded. 
So we develop a new approach by
exploring the structural relationship between the generator of the joint Markov process
and that of the diffusion models in \cref{L5.1}.
This is accomplished by matching the second order derivatives 
associated with the covariance of underlying Markov process
using the solutions of Poisson equations
which involve the difference between the 
coefficients of the original state process and those of the `averaged' diffusion models. 
In \cref{L5.2}, we also provide some crucial estimates for the
residual terms arising from the difference  of the two generators. 

Stability of switching diffusions has been studied extensively. 
Exponential stability for nonlinear Markovian switching diffusion processes 
has been studied in \cite{Mao99}, while $p$-stability and asymptotic stability 
for regime-switching diffusions have been addressed in \cite{Kha07}. 
For an underlying Markov process with a countable state space,
the stability of regime-switching 
diffusions has been considered in \cite{Shao14}. 
In these studies, the state and background Markov processes are unscaled, and
there is no `averaged' system. Under fast regime switching,  
we observe an `averaged' effect, 
and study how the ergodic properties of the `averaged' system
are related to those of the original system. 

\subsection{Organization of the paper}
The notation used in this paper is summarized in the next subsection.
In \cref{S2}, we describe the model of Markov-modulated general 
birth-death processes.
We present the results of exponential ergodicity and 
steady-state approximations in \cref{S2.1,S2.2}, respectively.
\cref{S3} contains some examples from queueing systems.
\cref{S4} is devoted to the proofs of \cref{T2.1,C2.1}.
The proofs of \cref{T2.2,C2.2} are given in \cref{S5}.
\cref{prop2.1} concerning the diffusion limit is given in \cref{AppA}.

\subsection{Notation}
We let $\NN$ and $\ZZ_+$ denote the set of natural numbers 
and the set of nonnegative integers, respectively. 
Let $\RR^d$ denote the set of $d$-dimensional real vectors, for $d\in\NN$.
The Euclidean norm and inner product in $\RR^d$ are 
denoted by $\abs{\,\cdot\,}$ and $\langle\,\cdot\,,\,\cdot\,\rangle$, respectively.
For $x\in\RR^d$, $x\transp$ denotes the transpose of $x$.
We denote the indicator function of a set $A\subset\RR^d$ by $\Ind_A$.
The minimum (maximum) of $a,b\in\RR$ is denoted by $a\wedge b\,(a\vee b)$,
and $a^{\pm} \df 0\vee(\pm a)$.
We let $e$ denote the vector in $\RR^d$ with all entries equal to $1$, and
$e_i$ denote the vector in $\RR^d$ with the $i^{\mathrm{th}}$ entry equal to $1$
and other entries equal to $0$.
The closure of a set $A\subset\RR^d$ is denoted by $\Bar{A}$.
The open ball of radius $r$ in $\RR^d$, centered at $x\in\RR^d$,
is denoted by $B_r(x)$.

For a domain $D\subset\RR^d$, the space $C^k(D)$ ($C^{\infty}(D)$) denotes 
the class of functions whose partial derivatives up to order $k$ (of any order) exist and
are continuous, and $C^k_b(D)$ stands for the functions in $C^{k}(D)$,
whose partial derivatives up to order $k$ are continuous and bounded.
The space $C^{k,1}(D)$ is the class of functions whose partial derivatives
up to order $k$ are Lipschitz continuous.
We let
$$
[f]_{2,1;D} \,\df\, \sup_{x,y\in D, x\neq y}
\frac{\babs{\grad^2 f(x) - \grad^2 f(y)}}{\abs{x - y}}
$$
for a domain $D\subset\RR^d$ and $f\in C^{2,1}(D)$.
For a nonnegative function $f\in C(\RR^d)$, 
we use $\order(f)$ to denote the space of function
$g\in C(\RR^d)$ such that
$\sup_{x\in\RR^d} \frac{\abs{g(x)}}{1 + f(x)} < \infty$.
By a slight abuse of notation,
we also let $\order(f)$ denote a generic member of this space.
Given any Polish space $\mathcal{X}$, we let $\cP(\mathcal{X})$
denote the space of probability measures on $\mathcal{X}$,
endowed with the Prokhorov metric.
For $\mu\in\cP(\mathcal{X})$ and a Borel measurable map 
$f\colon \mathcal{X}\mapsto \RR$,
we often use the simplified notation
$\mu(f) \df \int_{\mathcal{X}} f\,\D{\mu}$.

\section{Model and Results}\label{S2}

Let $Q= [q_{ij}]_{i,j\in\cK}$, with $\cK\df\{1,\dots,k_\circ\}$,
be an irreducible stochastic rate matrix,
and $\pi\df\{\pi_1,\dots,\pi_{k_\circ}\}$ denote its (unique)
invariant probability distribution.
We fix a constant $\alpha>0$.
For each $n\in N$, let $J^n$ denote the
finite-state irreducible continuous-time Markov chain
with state space $\cK$ and transition rate matrix $n^{\alpha}Q$.
In addition, for each $n\in\NN$ and $k\in\cK$, let
$\fX^n \subset\RR^d$ be a countable set, with no accumulation
points in $\Rd$, and
$R^n_k=\bigl[r^n_k(x,y)\bigr]_{x,y\in\fX^n}$ be a stochastic rate matrix
which gives rise to a nonexplosive, irreducible, continuous-time Markov chain.

The transition matrices $\{R^n_k\}$ satisfy the following structural assumptions.

\begin{hypothesis}\label{H1}
There exist positive constants $m_0$, $N_0$, and $C_0$, such that
the following hold for all $x\in\sX^n$, $n\in\NN$, and
$k\in\cK$.
\begin{enumerate}
\item[\textup{(a)}]
\emph{Bounded jumps.}
It holds that $r^n_k (x,x+z) =0$ for $\abs{z}>m_0$.

\smallskip
\item[\textup{(b)}]
\emph{Finitely many jumps.}
The cardinality of the
set $\sZ^n_k(x) \df \{z\in\Rd\colon r^n_k (x,x+z)>0\}$
does not exceed $N_0$.

\smallskip
\item[\textup{(c)}]
\emph{Incremental affine growth.}
It holds that
\begin{equation*}
\babs{r^n_k (x,x+z) - r^n_k (x',x'+z)}
\,\le\, C_0\bigl(n^{\nicefrac{\alpha}{2}}+\abs{x-x'}\bigr)\,.
\end{equation*}

\smallskip
\item[\textup{(d)}]
There exists some distinguished element $x^n_*\in\Rd$ such that
\begin{equation*}
r^n_k (x,x+z) \,\le\, C_0 (n^{\nicefrac{\alpha}{2}}+\abs{x-x^n_*} + n)\,.
\end{equation*}

\smallskip
\end{enumerate}
\end{hypothesis}

\Cref{H1} is assumed throughout the paper without further mention.
We refer the readers to 
\cref{Ex3.1,Ex3.4} 
for examples of verification of the conditions in (c) and (d).

\begin{remark}
The element $x^n_*\in\fX^n$ in part (d) plays an important role
in the analysis. 
For queueing models, $x^n_*$ may be chosen as steady state of the `average' fluid, 
referred to solutions of \cref{EA2.2A} later. 
\end{remark}

Consider the stochastic rate matrix $S^n$ on $\fX^n\times\cK$
whose elements are defined by
\begin{equation*}
s^n\bigl((x,i),(y,j)\bigr) \,\df\,
r^n_i(x,y) + n^{\alpha} q_{ij}\,,\quad x,y\in\fX^n\,,\ i,j\in\cK\,.
\end{equation*}
This defines a nonexplosive, irreducible
Markov chain $(X^n,J^n)$, where $J^n$ is as described
in the preceding paragraph.

In order to simplify some algebraic expressions,
we often use the notation
$\Tilde{r}^n_k(x,z) = r^n_k(x,x+z)$.

\begin{definition}\label{D2.1}
Let  $\beta\df\max\{\nicefrac{1}{2},1 - \nicefrac{\alpha}{2}\}$ be fixed. 
With $x^n_*$ as in \cref{H1}\,(d), we define the scaled process
\begin{equation*}
\widehat{X}^n\,\df\, \frac{X^n - {x}^n_{*}}{n^{\beta}}\,.
\end{equation*}
The state space of $\widehat{X}^n$ is given by 
\begin{equation*}
\widehat{\fX}^n \,\df\, \{\Hat{x}^n(x) \colon x\in\fX^n\}\,,
\end{equation*}
where $\Hat{x} = \Hat{x}^n(x) \df n^{-\beta}(x - x^n_*)$ for $x\in\RR^d$.
\end{definition}

Naturally, $(\widehat{X}^n,J^n)$ is a Markov process, and its
extended generator is given by
\begin{equation}\label{E-HcLn}
\widehat\cL^n f(\Hat{x},k) \,=\, \cL^n_k f(\Hat{x},k) 
+ \cQ^n f(\Hat{x},k)\,,\quad (\Hat{x},k)\in\widehat{\fX}^n\times\cK\,,
\end{equation}
for $f\in C_b(\RR^d\times\cK)$, where
\begin{equation}\label{E-cL}
\begin{aligned}
\cL^n_kf(\Hat{x},k) &\,\df\, \sum_{z\in\sZ^n(x)} 
\Tilde{r}^n_k(n^{\beta}\Hat{x} + {x}^n_*,z)
\bigl(f(\Hat{x} + n^{-\beta}z,k) - f(\Hat{x},k)\bigr)\,, \\
\cQ^nf(\Hat{x},k) &\,\df\, \sum_{\ell\in\cK}n^{\alpha}q_{k\ell}
\bigl(f(\Hat{x},\ell) - f(\Hat{x},k)\bigr)
\,=\, \sum_{\ell\in\cK}n^{\alpha}q_{k\ell}f(\Hat{x},\ell)\,.
\end{aligned}
\end{equation}
It is clear that $\widehat\cL^n f$
and $\cL^n_kf$ are well defined
for $f\in C_b(\RR^d)$, by viewing  $f$
as a function on $\Rd\times\cK$ which is constant with respect
to its second argument.

\subsection{Exponential ergodicity}\label{S2.1}
In this subsection, we provide a sufficient condition 
for the joint process $(\widehat{X}^n,J^n)$ to be exponentially ergodic.
We refer the reader to \cite{MT-III} for
the relevant Foster--Lyapunov criteria.
We introduce the following operator, 
which corresponds to the generator of  the `averaged' process.

\begin{definition}\label{D2.2}
Let
\begin{equation*}
\Bar{r}^n(x,z) \,\df\, \sum_{k\in\cK}\pi_k \Tilde{r}^n_k(x,z)\,,
\end{equation*}
and
\begin{equation}\label{E-sZn}
\sZ^n \,\df\, \cup_{x\in\fX^n} \cup_{k\in\cK} \sZ^n_k(x)\,.
\end{equation}

We define
$\overline\cL^n\colon C_b(\RR^d\times\cK)\mapsto C_b(\RR^d\times\cK)$ by
\begin{equation}\label{E-Lavg}
\overline\cL^n f(\Hat{x},k) \,\df\, 
\sum_{z\in\sZ^n}\Bar{r}^n(n^{\beta}\Hat{x} + {x}^n_*,z)
\bigl(f(\Hat{x} + n^{-\beta}z,k) - f(\Hat{x},k)\bigr)\,,
\quad (\Hat{x},k)\in\widehat{\fX}^n\times\cK\,,
\end{equation}
and $f\in C_b(\RR^d\times\cK)$.
\end{definition}

In the following theorem, we show that if $\overline\cL^n$
satisfies a Foster--Lyapunov inequality
with a suitable Lyapunov function,
then the original joint process $(\widehat{X}^n,J^n)$ is exponential ergodic.
The proof is given in \cref{S4}.

A function $f\in C(\RR^d)$ is called norm-like 
if $f(x) \rightarrow\infty$ as $\abs{x}\rightarrow\infty$.

\begin{theorem}\label{T2.1}
Suppose that 
there exist a sequence of nonnegative norm-like functions 
$\{\Lyap^n\in C(\RR^d)\colon n\in\NN\}$, $n_0\in\NN$, 
and some positive constants $\varepsilon_0$, $C$, $\overline{C}_1$, $\overline{C}_2$,  
not depending on $n$ such that  
\begin{equation}\label{ET2.1A}
\begin{aligned}
(1 + \abs{x})\,
\babs{\Lyap^n(x + y) - \Lyap^n(x)} &\,\le\, C\abs{y}\bigl(1 + \Lyap^n(x)\bigr)\,, \\
\bigl(1 + \abs{x}^2\bigr)\,
\babs{\Lyap^n\bigl(x + y + z\bigr) - \Lyap^n(x+ y) -
\Lyap^n(x+ z) + \Lyap^n(x)} &\,\le\, C\abs{y}\abs{z}\bigl(1 + \Lyap^n(x)\bigr)\,,   
\end{aligned}
\end{equation}
for any $y,z\in B_0(\varepsilon_0)\setminus\{0\}$,
$x\in\RR^d$ and $n\in\NN$,
and 
\begin{equation}\label{ET2.1B}
\overline\cL^n \Lyap^n(\Hat{x}) \,\le\, \overline{C}_1 - \overline{C}_2\Lyap^n(\Hat{x})
\qquad \forall\,\Hat{x}\in{\widehat{\fX}^n}\,,\  \forall\, n > n_0\,.
\end{equation}
Then, there exist a function $\widehat{\Lyap}^n\in C(\RR^d\times\cK)$,
and positive constants $\widehat{C}_1$, $\widehat{C}_2$,
and $n_1\in\NN$, such that, for all $n\ge n_1$, we have
\begin{equation}\label{ET2.1C}
\frac{1}{2}\bigl(\Lyap^n(\Hat{x}) - 1\bigr) \,\le\, \widehat{\Lyap}^n(\Hat{x},k) \,\le\,
\frac{3}{2}\Lyap^n(\Hat{x}) +\frac{1}{2}
\qquad \forall\,(\Hat{x},k)\in\widehat{\fX}^n\times\cK\,,
\end{equation}
and
\begin{equation}\label{ET2.1D}
\widehat\cL^n \widehat{\Lyap}^n(\Hat{x},k) \,\le\, \widehat{C}_1 - \widehat{C}_2
\widehat{\Lyap}^n(\Hat{x},k) \qquad \forall\,(\Hat{x},k)\in\widehat{\fX}^n\times\cK\,,
\ \forall\, n>n_1\,.
\end{equation}
As a consequence, 
$(\widehat{X}^n,J^n)$ is exponentially ergodic for all $n> n_1$,
and its invariant probability distributions are tight.
\end{theorem}

\begin{remark}
It follows from the proof of \cref{T2.1} that
$\widehat{C}_2$ can be selected arbitrarily close to $\overline{C}_2$,
so the rates of convergence of the `averaged' system and the Markov modulated
one become asymptotically close.
\end{remark}

\begin{remark}
A sufficient condition for a function $\Lyap^n\in C^{2,1}(\Rd)$ to satisfy
 \cref{ET2.1A} is
\begin{equation}\label{ER2.1A}
\abs{\grad \Lyap^n(x)} \,\le\, c\frac{1+\Lyap^n(x)}{1+\abs{x}}\,,
\quad\text{and}\quad
\abs{\grad^2 \Lyap^n(x)} + 
\bigl[\Lyap^n\bigr]_{2,1;B_{\epsilon}(x)}
\,\le\, c\frac{1+\Lyap^n(x)}{1+\abs{x}^2}\quad\forall\,x\in\Rd\,,
\end{equation}
for some fixed positive constants $\epsilon$ and $c$.
\end{remark}

In the next corollary, we relax the incremental growth hypothesis
in \cref{H1} (c). Its proof is contained in \cref{S4}. 
In \cref{Ex3.3}, we show that this result can be applied in the  
study of exponential ergodicity for Markov-modulated $M/M/n +M$ queues.


We replace \cref{H1}\,(c) by the following weaker assumption.

\begin{assumption}\label{A2.1}
Suppose that \cref{H1}\,(a), (b) and (d) are satisfied,
and $\Tilde r^n_k$ can be decomposed into 
\begin{equation*}
\Tilde{r}^n_k(x,z) \,=\, \phi^n_k(x,z) + \psi^n_k(x,z)\,, 
\quad x\in\fX^n\,, \quad z\in\sZ^n_k(x)\,,
\end{equation*}
where $\phi^n_k(x,z)$ and $\psi^n_k(x,z)$, $k\in\cK$, 
are locally bounded functions on $\fX^n\times\sZ^n$.
In addition, 
using without loss of generality the same constant, 
there exist $\delta_1,\delta_2 \in[0,1]$
such that
\begin{equation}\label{A2.1B}
\abs{\psi^n_k(x,z) -  \psi^n_k(y,z)}
\,\le\, C_0\bigl(n^{\nicefrac{\alpha}{2}} + \abs{x - y}^{\delta_1}\bigr)\
\qquad \forall\,k\in\cK\,, \ \forall\,x,y\in\fX^n\,, \ \forall\,z\in\sZ^n\,,
\end{equation}
and
\begin{equation}\label{A2.1C}
\abs{\psi^n_k(x,z)} \,\le\, C_0\bigl(n^{\nicefrac{\alpha}{2}} + \abs{x - x^n_*}^{\delta_2}
+ n \bigr) \qquad \forall\,k\in\cK\,,
\ \forall\,(x,z)\in\fX^n\times\sZ^n\,,
\end{equation}
with $x^n_*\in\RR^d$ as in \cref{H1} (d), and for $n\in\NN$.
\end{assumption}

\begin{corollary}\label{C2.1}
Grant \cref{A2.1}.
Let
$\cG^n_k\colon C_b(\RR^d\times\cK)\mapsto C_b(\RR^d\times\cK)$
be defined by
\begin{equation}\label{EC2.1A}
\cG^n_k f(\Hat{x},k) \,\df\, 
\sum_{z\in\sZ^n}\bigl(\phi^n_k(n^{\beta}\Hat{x} + {x}^n_*,z)  
+ \Bar{\psi}^n(n^{\beta}\Hat{x} + {x}^n_*,z)\bigr)
\bigl(f(\Hat{x} + n^{-\beta}z,k) - f(\Hat{x},k)\bigr)
\end{equation}
for $(\Hat{x},k)\in\widehat{\fX}^n\times\cK$ and $f\in C_b(\RR^d\times\cK)$, and
with $x^n_*$ as in \cref{A2.1},
where $\Bar{\psi}^n(x,z) \df \sum_{k\in\cK}\pi_k\psi^n_k(x,z)$.
Suppose that \cref{ET2.1A} holds with the second inequality replaced by
$$
\bigl(1 + \abs{x}^{1+\delta_2}\bigr)\,
\babs{\Lyap^n\bigl(x + y + z\bigr) - \Lyap^n(x+ y) -
\Lyap^n(x+ z) + \Lyap^n(x)} \,\le\, C\abs{y}\abs{z}\bigl(1 + \Lyap^n(x)\bigr)\,, 
$$
where $\delta_2$ is as in \cref{A2.1}, and 
there exist $n_2\in\NN$ and some positive constants $C_1$ and $C_2$ such that
\begin{equation}\label{EC2.1B}
\cG^n_k \Lyap^n(\Hat{x}) \,\le\, C_1 - C_2\Lyap^n(\Hat{x})
\qquad \forall\,(\Hat{x},k)\in\widehat{\fX}^n\times\cK\,, \  \forall\, n > n_2\,.
\end{equation}
Then, the results in \cref{ET2.1C,ET2.1D} hold.
\end{corollary}

\begin{remark}\label{R2.3}
If we replace \cref{H1}\,(c) and (d) by
\begin{equation}\label{ER2.3A}
\babs{r^n_k (x,x+z) - r^n_k (x',x'+z)}
\,\le\, C_0\bigl(1+\abs{x-x'}\wedge n\bigr)\,,
\end{equation}
and
\begin{equation}\label{ER2.3B}
r^n_k (x^n_*,x^n_*+z)
\,\le\, C_0 n\,,
\end{equation}
respectively,
and \cref{ET2.1A} by
\begin{equation}\label{ER2.3C}
\begin{aligned}
\babs{\Lyap^n(x + y) - \Lyap^n(x)} &\,\le\, C\abs{y}\bigl(1 + \Lyap^n(x)\bigr)\,, \\
\babs{\Lyap^n\bigl(x + y + z\bigr) - \Lyap^n(x+ y) -
\Lyap^n(x+ z) + \Lyap^n(x)} &\,\le\, C\abs{y}\abs{z}\bigl(1 + \Lyap^n(x)\bigr)\,,   
\end{aligned}
\end{equation}
then, provided $\beta$ and $\alpha$ satisfy $2\beta+\alpha>2$,
the conclusion of \cref{T2.1} still holds.
This can be easily seen from the proof in \cref{S4}.
Note that \cref{ER2.3C} is satisfied for exponential functions.

The transition rates of multiclass
$M/M/n$ queues, that is,
the model in \cref{Ex3.3} with no abandonment ($\gamma_i(k)\equiv0$), satisfy
\cref{ER2.3A,ER2.3B}. 
Uniform exponential ergodicity of this model (with spare capacity, or equivalently, positive safety staffing) is established
in \cite{AHP-18} using exponential Lyapunov functions.
Thus, we may use exponential Lyapunov functions in \cref{ET2.1B},
and take advantage of the results in \cite{AHP-18} to establish
exponential ergodicity of Markov-modulated multiclass $M/M/n$ queues with positive safety staffing
using the Lyapunov functions in \cite{AHP-18}.
We leave it to the reader to verify that for $\alpha\ge1$, we can
in fact establish uniform exponential ergodicity over all work-conserving
scheduling policies.
For $\alpha<1$ the discontinuity allowed in the policies need to be
restricted.

Extending this to the classes of multiclass multi-pool models
studied in \cite{HAP-19} is also possible.
\end{remark}

\subsection{Steady-state approximations}\label{S2.2}

Here, we use a function
$\xi^n_z(x,k)$ for $(x,z)\in\Rd\times\Rd$ and $k\in\cK$
which interpolates the transition rates
in the sense that
\begin{equation*}
\xi^n_{z}(x,k) \,=\, r^n_k(x,x+z) \quad \text{if } x,x+z\in\fX^n\,.
\end{equation*}
Recall the definition of $\sZ^n$ in \cref{E-sZn}.
It is clear that for $z\notin\sZ^n$ we may let $\xi^n_z\equiv0$.
Thus
\begin{equation*}
\sZ^n \,=\, \bigl\{z\in\Rd\colon \exists\, x, k \text{\ such that\ }
\xi^n_{z}(x,k)>0\}\,.
\end{equation*}
This of course also implies that
\begin{equation}\label{E-m0}
\xi^n_{z}(x,k) \,=\, 0 \quad\text{if\ } \abs{z}> m_0
\end{equation}
by \cref{H1}\,(a).

We let $\cI\df\{1,\dotsc,d\}$,
and define
\begin{equation}\label{E-Lambda}
\begin{aligned}
\Xi^n(x,k) &\,\df\, \sum_{z\in\sZ^n} z\,\xi^n_{z}(x,k)\,,\\
\Gamma^n_{ij}(x,k) &\,\df\, \sum_{z\in\sZ^n}z_iz_j\xi^n_z(x,k)\,,
\quad  i,j\in\cI\,,
\end{aligned}
\end{equation}
for $(x,k)\in \RR^d\times\cK$.

We impose the following structural assumptions on the function
$\xi^n$.

\begin{assumption}\label{A2.2}
The following hold.
\begin{enumerate}
\item[\textup{(i)}]
The cardinality of the
set $\{z\in\Rd\colon \xi^n_{z}(x,k)>0\}$
does not exceed $\widetilde{N}_0$.

\smallskip
\item[\textup{(ii)}]
For each $n\in\NN$, there exists 
$x^n_*\in\Rd$ satisfying
\begin{equation}\label{EA2.2A}
\sum_{k\in\cK}\pi_k\Xi^n({x}^n_*,k)\,=\, 0\,.
\end{equation}

\smallskip
\item[\textup{(iii)}] 
The function
$\xi^n_z$ is uniformly Lipschitz continuous in its first argument, 
that is,
there exists some positive constant $\widetilde{C}$ such that 
\begin{equation}\label{EA2.2B}
\abs{\xi^n_z(x,k) -  \xi^n_z(y,k)}
\,\le\, \widetilde{C}\abs{x - y} \qquad \forall\,k\in\cK\,,
\ \forall\,x,y\in\RR^d\,, \ \forall\,z\in\sZ^n\,, 
\end{equation}
for all $n\in\NN$. In addition,
using without loss of generality the same constant, we assume that
\begin{equation}\label{EA2.2C}
\max_{z\in\Rd}\,\xi^n_z(x^n_*, k)\,\le\,\widetilde{C} n
\qquad  \forall\, k\in\cK\,,\ \forall\,n\in\NN\,.
\end{equation}

\smallskip
\item[\textup{(iv)}]
The matrix $\Gamma^n(x^n_*,k)$ is positive definite, and
\begin{equation}\label{EA2.2D}
\frac{1}{n}\, \Gamma^n(x^n_*, k) \,\xrightarrow[n\to\infty]{}\, \Bar{\Gamma}(k)\,,
\end{equation}
where $\Bar{\Gamma}(k)$ is also a positive definite $d\times d$ matrix, for
all $k\in\cK$.

\smallskip
\end{enumerate}
\end{assumption}

We note here that the nondegeneracy hypothesis
in \cref{A2.2}\,(iv) is used in \cref{L5.3} to derive gradient estimates of 
the solution of a Poisson equation.

\begin{remark}
\Cref{EA2.2B} is of course much stronger than \cref{H1}\,(c).
This is needed for the results in this section which rely on
certain Schauder estimates for solutions of the Poisson equation for the
generator of an approximating diffusion equation. 
\end{remark}

Let $\{A^n_z \colon z\in \sZ^n\}$ be a family
of independent unit rate Poisson processes,
and $\Tilde{A}^n_{z}(t) \df A^n_{z}(t) - t$.
Then, the $d$-dimensional process $X^n(t)$ is governed by the equation 
\begin{equation}\label{S2A}
\begin{aligned}
X^n(t) &\,=\, X^n(0) + 
\sum_{z\in\sZ^n} z\, A^n_{z}
\biggl(\int_0^t \xi^n_{z}\bigl(X^n(s),J^n(s)\bigr)\,\D{s}\biggr) 
\nonumber\\ &\,=\, X^n(0) + M^n(t) +
\int_0^t \Xi^n\bigl(X^n(s),J^n(s)\bigr)\,\D{s} \,,
\end{aligned}
\end{equation}
where
\begin{equation*}
M^n(t)\,\df\, 
\sum_{z\in\sZ^n}z \,\Tilde{A}^n_{z}
\biggl(\int_0^t\xi^n_{z}\bigl(X^n(s),J^n(s)\bigr)\,\D{s}\biggr)\,.
\end{equation*}
Note that $M^n(t)$ is a local martingale  
with respect to the filtration
$$
\mathcal{F}^n_t \,\df\, \upsigma\left\{X^n(0),J^n(s),
\Tilde{A}^n_{z}\biggl(\int_0^t\xi^n_{z}\bigl(X^n(s),J^n(s)\bigr)\,\D{s}\biggr),  
\int_0^t\xi^n_{z}\bigl(X^n(s),J^n(s)\bigr)\,\D{s}\colon z\in\sZ^n, s\le t\right\}
\,.
$$
The locally predictable quadratic variation of $M^n$ satisfies 
$$
\langle M^n  \rangle (t) \,=\, \int_0^t \Gamma^n\bigl(X^n(s),J^n(s)\bigr)\,\D{s}\,,
\quad t\ge0\,,
$$
where the function $\Gamma^n = [\Gamma^n_{ij}] 
\colon \RR^d\times\cK \mapsto \RR^{d\times d}$ is given
in \cref{E-Lambda}.

By \cref{EA2.2B}, it is evident that given 
$x^n(0) \in\RR^d$, 
there exists a unique solution $x^n(t)$ satisfying
\begin{equation*}
x^n(t) \,=\, x^n(0) + \sum_{k\in\cK}\pi_k\int_0^t
\Xi^n(x^n(s),k)\,\D{s} \,.
\end{equation*}
We refer to this as the $n^{\text{th}}$ `averaged' fluid model.


In this section, the scaled process is defined as in \cref{D2.1}, with the
exception that $x^n_*\in\Rd$ is specified in \cref{A2.2}.
Note that in the extended generator in \cref{E-HcLn,E-cL} we may replace
$\Tilde{r}^n_k(n^{\beta}\Hat{x} + {x}^n_*,z)$ by $\xi^n_z(n^{\beta}\Hat{x} + {x}^n_*, k)$.
It is evident from \cref{S2A}, that $\widehat{X}^n$ satisfies
\begin{equation}\label{E-hatX}
\widehat{X}^n(t) \,=\, \widehat{X}^n(0) + \widehat{M}^n(t)
+ \int_0^t\widehat{\Xi}^n\bigl(\widehat{X}^n(s),J^n(s)\bigr)\,\D{s}\,,
\end{equation}
where
\begin{equation}\label{E-hatXi}
\widehat{M}^n\,\df\, \frac{M^n}{n^{\beta}}\,,
\qquad \text{and} \qquad
\widehat{\Xi}^n(\Hat{x},k)\,\df\, 
\frac{\Xi^n(n^{\beta}\Hat{x} + {x}^n_{*},k)}{n^{\beta}}\,,
\quad (\Hat{x},k)\in\RR^d\times\cK\,.
\end{equation}
The locally predictable quadratic variation of $\widehat{M}^n$ is given by
$$
\langle \widehat{M}^n \rangle (t) \,=\, \int_0^{t}
\Bar{\Gamma}^n\bigl(\widehat{X}^n(s),J^n(s)\bigr)\,\D{s}\,,\quad t\ge0\,,
$$
with
\begin{equation}\label{E-bargamma}
\Bar{\Gamma}^n(\Hat{x},k) \,\df\, \frac{1}{n^{2\beta}}\,
\Gamma^n(n^{\beta}\Hat{x} + {x}^n_*,k)\,,
\qquad (\Hat{x},k)\in\RR^d\times\cK\,.
\end{equation}

We next introduce a sequence of processes that approximate $\widehat{X}^n$.
Let $\widehat{Y}^n$ be the strong solution to the It\^o $d$-dimensional stochastic
differential equation (SDE)
\begin{equation}\label{E-sde}
\D \widehat{Y}^n(t) \,=\, \Bar{b}^n\bigl(\widehat{Y}^n(t)\bigr)\,\D{t} 
+ \sigma^n \D W(t)\,, 
\end{equation}
with $\widehat{Y}^n(0) = y_0$,  
where $W(t)$ is a $d$-dimensional standard Brownian motion,
$$\Bar{b}^n_i(\Hat{y})\df \sum_{k\in\cK}\pi_k\widehat{\Xi}^n_i(\Hat{y},k)\,,
\qquad \,\Hat{y}\in\RR^d\,, \ i\in\cI\,,
$$
with $\widehat{\Xi}^n$ defined in \cref{E-hatXi}.
The diffusion matrix $\sigma^n$ is characterized as follows.
Let $\Upsilon = (\Pi - Q)^{-1} - \Pi$ denote the deviation matrix
corresponding to the transition rate matrix $Q$ \cite{CD02}.
Let $\Theta^n = [\theta^n_{ij}]$ be defined by
\begin{equation}\label{ES2.2B}
\theta^n_{ij}\,\df\, 2\sum_{\ell\in\cK}\sum_{k\in\cK} 
\frac{\Xi^n_i({x}^n_*,k)\Xi^n_j({x}^n_*,\ell)}{n^{\alpha + 2\beta}}
\pi_k\Upsilon_{k\ell} \,,\quad i,j\in\cI\,,
\end{equation}
and
$$\Bar{a}^n(x) \,=\, [\Bar{a}^n_{ij}](x) 
\,\df\, \sum_{k\in\cK}\pi_k\Bar{\Gamma}^n(x,k)\,, \quad x\in\RR^d\,.$$
Then, $\sigma^n$ satisfies
\begin{equation}\label{E-Sigma}
\Sigma^{n} \df (\sigma^n)\transp\sigma^n\,=\, \Bar{a}^n(0) + \Theta^n\,.
\end{equation}


The generator of $\widehat{Y}^n$ denoted by $\sA^n$ is given by
\begin{equation}\label{E-sAn}
\sA^n f(x) \,=\, \sum_{i\in\cI}\Bar{b}^n_i(x)\,\partial_i f(x) +
\frac{1}{2}\sum_{i,j\in\cI}\Sigma^{n}_{ij}\,\partial_{ij}f(x)\,,
\quad f\in C^2(\RR^d)\,.
\end{equation}

We borrow the following definitions from \cite{Gurvich14}.
We say that a function $f\in C^2(\RR^d)$ is sub-exponential
if $f \ge 1$ and there exists some positive constant $c$
such that
$$
\abs{\grad f(x)} + \babs{\grad^2 f(x)} \,\le\, c\,\E^{c\abs{x}} \qquad
\forall\,x\in\RR^d\,,
$$
and
$$
\sup_{\{z\colon \abs{z}\le1\}}\frac{f(x+z)}{f(x)} \,\le\, c \qquad
\forall\,x\in\RR^d\,.
$$ 
We also let $\sB_x$ denote the open ball around $x\in\Rd$ of
radius $(1+\abs{x})^{-1}$, and define
$$
\norm{f}_{C^{0,1}(\sB_x)} \,\df\, \sup_{y\in\sB_x}\abs{f(x)}
+ \sup_{y,z\in\sB_x}\frac{\abs{f(y)-f(z)}}{\abs{y-z}}\,,\quad f\in C^{0,1}(\Rd)\,.
$$

The following assumption concerning the ergodic properties
of $\widehat{Y}^n$ 
plays a crucial role in the proofs for steady-state approximations.

\begin{assumption}\label{A2.3}
There exist a sub-exponential norm-like function $\sV\in C^2(\RR^d)$, 
a positive constant $\kappa$, and an open ball $\sB$ such that 
$$
\sA^n \sV(x) \,\le\, \Ind_{\sB}(x) - \kappa\sV(x) \qquad 
\forall\, x\in\RR^d\,,\ \forall\,n\in\NN\,.
$$
\end{assumption}

We continue with the main result of this section.
Its proof is given in \cref{S5}.
Let $\nu^n\in\cP(\RR^d)$ denote the steady-state distribution of $\widehat{Y}^n$.

\begin{theorem}\label{T2.2}
Grant  \cref{A2.2,A2.3}. Assume that $(\widehat{X}^n,J^n)$ is ergodic,
and its steady-state distribution $\uppi^n\in\cP(\RR^d\times\cK)$ satisfies
\begin{equation}\label{ET2.2A}
\limsup_{n\rightarrow\infty}\int_{\RR^d\times\cK}
\sV(\Hat{x})(1 + \abs{\Hat{x}})^5\uppi^n(\D{\Hat{x}},\D{k})\,<\,\infty\,.
\end{equation}
Then, for any $f \colon \RR^d \mapsto \RR$ such that
$\norm{f}_{C^{0,1}(\sB_x)}\le \sV(x)$, and $\alpha >0$, we have
\begin{equation}\label{ET2.2B}
\abs{\uppi^n(f) - \nu^n(f)} \,=\, 
\order\biggl(\frac{1}{n^{\nicefrac{\alpha}{2}\wedge \nicefrac{1}{2} }}\biggr)\,.
\end{equation}
\end{theorem}

The order of the function in \cref{ET2.2A} 
is determined by the estimates in \cref{L5.2}, and
the gradient estimates of the solutions to the Poisson equation in \cref{L5.3}.
In the following corollary, we provide a sufficient condition for \cref{ET2.2A}.
We give its proof in \cref{S5}.
In \cref{S3}, we show that this sufficient condition holds in many examples.

\begin{corollary}\label{C2.2}
Grant \cref{A2.2}. Let $\sV$ and $\widetilde{\sV}$
be two sub-exponential functions in $C^2(\RR^d)$ satisfying 
\cref{A2.3} such that 
\begin{equation}\label{EC2.2A}
\sV(x)(1 + \abs{x}^5) \,\le\, \widetilde{\sV}(x)\,,
\end{equation}
and 
\begin{equation}\label{EC2.2B}
(1 + \abs{x})\bigl( \abs{\grad{\widetilde{\sV}(x)}} 
+ \babs{\grad^2{\widetilde{\sV}(x)}} \bigr)
+ (1 + \abs{x}^2)\bigl[\widetilde{\sV}\bigr]_{2,1;B_{\nicefrac{m_0}{n^{\beta}}}(x)}
\,\le\, C \widetilde{\sV}(x)\,,
\end{equation}
for some positive constant $C$ and any $x\in\RR^d$,
and with $m_0$ as in \cref{E-m0}.
Then
\cref{ET2.2A} holds for $\sV$.
As a consequence, \cref{ET2.2B} holds. 
\end{corollary}

\section{Examples}\label{S3}
In this section, we demonstrate how the
results of \cref{S2} can be applied through examples.

\begin{example}[Markov-modulated $M/M/\infty$ queue] \label{Ex3.1}
We consider a process given by
\begin{equation*}
{X}^n(t) \,\df\, {X}^n(0)
+ A^{n}_1\biggl(\int_0^{t}n\lambda\bigl(J^n(s)\bigr)\,\D{s}\biggr)
- A^{n}_{-1}\biggl(\int_0^{t}\mu\bigl(J^n(s)\bigr) X^n(s)\,\D{s}\biggr) \,,
\end{equation*}
where  
$ A^{n}_{-1}$ and $A^{n}_1$ are mutually independent Poisson processes with rate one,
for $n\in\NN$. We assume that 
$\lambda(k) > 0$ 
and  $\mu(k) > 0$, for $k\in\cK$.
We let
\begin{equation}\label{Ex3.2A}
{x}^n_{*} \,=\, n\frac{\sum_{k\in\cK}\pi_k\lambda(k)}{\sum_{k\in\cK}\pi_k\mu(k)}\,.
\end{equation} 
Recall that $\widehat{X}^n = n^{-\beta}(X^n - x^n_*)$, 
and then $\widehat{\fX}^n = \{\Hat{x}^n(x) \colon x\in\ZZ_+ \}$. 
It is evident that $\lambda(k)$ and $\mu(k)x$ satisfy \cref{H1}~(c) and (d).
Let $\Bar{\lambda} \df \sum_{k\in\cK}\pi_k\lambda(k)$ 
and $\Bar{\mu} \df \sum_{k\in\cK}\pi_k\mu(k)$.
By \cref{D2.2}, we obtain 
\begin{equation}\label{Ex3.2B}
\overline\cL^n f(\Hat{x}) \,=\,n\Bar{\lambda}
\,\bigl(f(\Hat{x} + n^{-\beta}) - f(\Hat{x})\bigr) 
+ \Bar{\mu}\,(n^{\beta}\Hat{x} + x^n_*)
\bigl(f(\Hat{x} - n^{-\beta}) - f(\Hat{x})\bigr)
\qquad \forall\,\Hat{x}\in\widehat{\fX}^n\,.
\end{equation}
Let $\Lyap(x) = \abs{x}^m$, for $x\in\RR$, with even integer $m \ge 2$. 
It is clear that 
\begin{equation}\label{Ex3.1C}
\abs{\Hat{x} \pm n^{-\beta}}^m - \abs{\Hat{x}}^m \,=\, \pm n^{-\beta}m(\Hat{x})^{m-1} 
+ \order(n^{-2\beta})\order(\abs{\Hat{x}}^{m-2})\,.
\end{equation}
Thus we obtain from \cref{Ex3.2A,Ex3.2B} that 
\begin{equation*}
\begin{aligned}
\overline\cL^n \Lyap(\Hat{x}) &\,=\, n\Bar{\lambda}
\bigl(\abs{\Hat{x}+n^{-\beta}}^m - \abs{\Hat{x}}^m - n^{-\beta}m\abs{\Hat{x}}^{m-1}\bigr)
+ \Bar{\mu} n^{\beta} \Hat{x} 
\bigl(\abs{\Hat{x}-n^{-\beta}}^m - \abs{\Hat{x}}^m\bigr)  \nonumber \\
&\qquad + \Bar{\mu} {x}^n_{*}
\bigl(\abs{\Hat{x}-n^{-\beta}}^m - \abs{\Hat{x}}^m 
+ mn^{-\beta}\abs{\Hat{x}}^{m-1}\bigr) \nonumber \\
&\,=\, \Bar{\lambda}\order(n^{1-2\beta})\order(\abs{\Hat{x}}^{m-2}) 
+ \Bar{\mu}
\bigl(-\abs{\Hat{x}}^m + \order(n^{-\beta})\order(\abs{\Hat{x}}^{m-1}) 
+ \order(n^{1-2\beta})\order(\abs{\Hat{x}}^{m-2}) \bigr) \nonumber\\
&\,\le\, C_1 - C_2\Lyap(\Hat{x}) \qquad \forall\,\Hat{x}\in\widehat{\fX}^n\,,
\end{aligned}
\end{equation*}
for some positive constants $C_1$ and $C_2$,
where in the second equality we use \cref{Ex3.1C}, 
and in the last line we apply Young's inequality.
It is straightforward to verify that $\Lyap(x)$ satisfies \cref{ER2.1A}.
Therefore, the assumptions in \cref{T2.1} hold, 
and $(\widehat{X}^n,J^n)$ is exponentially ergodic for all large enough $n$.

Next we verify the assumptions in \cref{C2.2}.
The equation in \eqref{EA2.2A} becomes
\begin{equation}\label{Ex3.2D}
\sum_{k\in\cK}\pi_k\Xi^n(x^n_*,k) \,=\, 
\sum_{k\in\cK}\pi_k n\lambda(k)
- \sum_{k\in\cK}\pi_k\mu(k)x^n_* \,=\, 0\,.
\end{equation}
Note that $x^n_*$ in \eqref{Ex3.2A} is the unique solution to \cref{Ex3.2D}.
Recall the representation of $\widehat{Y}^n$ in \cref{E-sde}. 
In this example, it follows by \cref{Ex3.2D} that
$$
\Bar{b}^n(x) \,=\, n^{-\beta}\Bar{\mu}x^n_* - n^{-\beta}\Bar{\mu} (n^{\beta}x + x^n_*)
\,=\, -\Bar{\mu}x \qquad \forall\,x\in\RR\,,
$$
and
$$
\Bar{a}^n(0) \,=\, n^{-2\beta}(n\Bar{\lambda} + \Bar{\mu}x^n_*) \,=\, 
n^{1-2\beta}2\Bar{\lambda}\,.
$$
Let $\sV(x) = \kappa + \abs{x}^m$, with $\kappa \ge 1$ for some even integer $m\ge 2$.
Then, \cref{A2.2,A2.3} are satisfied, and the result
in \cref{C2.2} follows.
\end{example}

The following example concerns  Markov-modulated multiclass $M/M/N + M$ queues.
Exponential ergodicity for these queues under a static priority scheduling policy
has been studied in \cite[Theorem 4]{ADPZ19},
which treats a special case of the model considered in this paper.
Here we show that by using the result in \cref{C2.1},
the proof of \cite[Theorem 4]{ADPZ19} is simplified a lot.
We also extend the results in \cite[Theorem 4 and Lemma~3]{ADPZ19} to include a larger 
class of scheduling policies such that 
the Markov-modulated queues have exponential ergodicity.

\begin{example}
[Markov-modulated multiclass $M/M/N + M$ queues]\label{Ex3.3}
We consider a $d$-dimensional birth-death process $\{X^n(t)\colon t\ge0\}$,
with state space $\ZZ^d_+$, given by
\begin{equation*}
\begin{aligned}
{X}^n_i(t) &\,\df\, {X}^n_i(0)
+ A^{n}_{e_i}\biggl(\int_0^{t}n\lambda_i\bigl(J^n(s)\bigr)\,\D{s}\biggr) \\
&\qquad- A^{n}_{-e_i}\biggl(\int_0^{t}\Bigl(\mu_i\bigl(J^n(s)\bigr)z^n_i(X^n(s))
+ \gamma_i\bigl(J^n(s)\bigr)\bigl(X^n_i(s) - z^n_i(X^n(s))\bigr)\Bigr)\,\D{s}\biggr)
\end{aligned}
\end{equation*}
for $i\in\cI\df\{1,\dots,d\}$, 
where $\{A^{n}_{e_i},A^{n}_{-e_i}\colon i\in\cI\}$
are mutually independent Poisson processes with rate $1$, and
$z^n$ is the static priority policy defined by
$$
z^n_i(x) \,\df\, x_i\wedge\biggl(n - \sum_{j=1}^{i-1}x_j\biggr)^+
\qquad \forall\,i\in\cI\,.
$$
We assume that $\{\lambda_i(k),\mu_i(k),\gamma_i(k)\colon i\in\cI,k\in\cK\}$
are strictly positive, and the system is critically loaded, that is, 
$\sum_{i\in\cI}\rho_i = 1$ with $\rho_i \df \nicefrac{\Bar{\lambda}_i}{\Bar{\mu}_i}$. 
\Cref{EA2.2A} becomes 
$$
\sum_{k\in\cK}\pi_k \Xi^n_i({x}^n_*,k)
\,=\, n\Bar{\lambda}_i - \Bar{\mu}_iz^n_i(x^n_*) 
- \Bar{\gamma}_i(x^n_{*,i} - z^n_i(x^n_*)) \,=\, 0
\qquad \forall\,i\in\cI\,,
$$
which has a unique solution $x^n_* = n\rho$
with $\rho = (\rho_1,\dots,\rho_d)$.

We  first establish exponential ergodicity 
and verify the assumptions in \cref{A2.1}. 
Let
$$
\psi^n_{e_i}(x,k) \,=\, n \lambda_i(k)\,,
\qquad \psi^n_{-e_i}(x,k) \,=\, n \rho_i \mu_i(k)\,,
$$
and
$$
\phi^n_{-e_i}(x,k) \,=\, \mu_i(k)(z^n_i(x) - n\rho_i) 
+ \gamma_i(k)(x_i - z^n_i(x)) 
$$ 
for $i\in\cI$ and $(x,k)\in\RR^d\times\cK$.
Then, $\Bar{\psi}^n_{e_i}(x) = n\Bar{\lambda}_i$
and $\Bar{\psi}^n_{-e_i}(x) = n\rho_i\Bar{\mu}_i = n\Bar{\lambda}_i$.
It is evident that the functions $\psi^n_{e_i}$
and $\psi^n_{-e_i}$ satisfy \cref{A2.1B,A2.1C}.
Note that $z_i^n(x) \le x_i$, and thus \cref{H1} (d) is satisfied. 
Let $\Lyap_{\zeta,m}(x) \df \sum_{i\in\cI}\zeta_i\abs{x_i}^m$ 
for $x\in\RR^d$, even integer $m\ge 2$, and a positive vector $\zeta\in\RR^d$
to be chosen later.
Recall $\cG^n_k$ in \cref{EC2.1A}.
It is straightforward to verify that 
\begin{equation*}
\begin{aligned}
\cG^n_k\Lyap_{\zeta,m}(\Hat{x}) 
&\,=\, n^{-\beta}\sum_{i\in\cI} - \phi^n_{-e_i}(n^{\beta}\Hat{x}+ n\rho,k)
\lambda_i\abs{\Hat{x}_i}^{m-1} \\
&\qquad + n^{-2\beta}\sum_{i\in\cI}
\bigl(2n\Bar{\lambda}_i + \phi^n_{-e_i}(n^{\beta}\Hat{x}+ n\rho,k)\bigr)
\order(\abs{\Hat{x}_i}^{m-2})\,.
\end{aligned}
\end{equation*}
Since $\inf_{i,k}\{\mu_i(k),\gamma_i(k)\} > 0$, 
it follows by \cite[Lemma~5.1]{ABP15} that there exist some 
positive vector $\lambda$, $n_0\in\NN$, and positive constants $C_1$ and
$C_2$ such that
\begin{equation}\label{Ex3.3A}
\cG^n_k\Lyap_{\zeta,m}(\Hat{x}) \,\le\, C_1
- C_2\Lyap_{\zeta,m}(\Hat{x})\,, 
\quad (x,k)\in\widehat{\fX}^n\times\cK\,,
\quad n\ge n_0\,.
\end{equation}
Therefore, the result in \cref{C2.1} follows.
We remark that the claim in \cref{C2.1} holds for any work-conserving scheduling policy 
satisfying \cref{Ex3.3A}, since there is no continuity assumption on $\phi^n_{-e_i}$.
This extends the results of \cite[Theorem~4 and Lemma~3]{ADPZ19}.
Indeed the proofs of these results can be simplified a lot
following the approach above,
since we only need to consider the constant functions $\psi^n_{e_i}$ and $\psi^n_{-e_i}$
in $x$.

Next we focus on steady-state approximations for this example.
It is straightforward to verify that
the coefficients in \cref{E-sde} take the form 
\begin{equation}\label{Ex3.3B}
\begin{aligned}
\Bar{b}^n_i(x) \,=\, -\frac{\Bar{\mu}_i}{n^{\beta}}
\bigl(z^n_i(n^{\beta}x + x^n_{*}) - z^n_i(x^n_*)\bigr)
- \frac{\Bar\gamma_i}{n^{\beta}}\bigl(n^{\beta}x_i 
- (z^n_i(n^{\beta}x + x^n_{*}) - z^n_i(x^n_*)) \bigr)\,,
\quad i\in\cI\,,
\end{aligned}
\end{equation}
and
$$
\Bar{a}^n_{ii}(0) \,=\, \frac{1}{n^{2\beta}}
\bigl(n\Bar{\lambda}_i + \Bar{\mu}_i z^n_i({x}^n_*) 
+ \Bar{\gamma}_i\bigl(x^n_{*,i} - z^n_i(x^n_*)\bigr)
\,=\, n^{1-2\beta}2\Bar{\lambda}_i\,,
\quad \forall\,i\in\cI\,,
$$
and that $\Bar{a}^n_{ij}(0) = 0$ for $i\neq j$.
We let $\sV_{\zeta,m}(x) = \kappa + \sum_{i\in\cI}\zeta_i\abs{x_i}^m$ 
 for some positive 
vector $\zeta\in\RR^d$, an even integer $m\ge 2$, and $\kappa \ge 1$. 
Repeating the calculation in \cite[Lemma 5.1]{ABP15},
it follows that there exist some positive vector $\zeta\in\RR^d$ and
some positive constants $c_1$ and $c_2$ such that
$$
\langle \Bar{b}^n(x), \grad\sV_{\zeta,m}(x)\rangle 
\,\le\, c_1 - c_2\sV_{\zeta,m}(x)
\qquad \forall\,x\in\RR^d\,.
$$
It follows directly by Young's inequality that there exists some positive constant
$c_3$ such that
$$
\babs{\grad^2\sV_{\zeta,m}(x)} \,\le\, c_3 - \frac{c_2}{2}\sV_{\zeta,m}(x)
\qquad \forall\,x\in\RR^d\,.
$$
Thus, we have verified \cref{A2.3}. 
Since $z^n_i$ is Lipschitz continuous, it is evident that \cref{A2.2} holds.
As a result, \cref{C2.2} follows.

When $d=1$, \cref{EA2.2A} becomes
$$
\sum_{k\in\cK}\pi_k \Xi^n({x}^n_*,k) \,=\, 
n\Bar{\lambda} -  \Bar{\mu}\bigl({x}^n_* \wedge n\bigr) 
- \Bar{\gamma}\bigl({x}^n_* - n)^+ \,=\, 0\,,
$$
which can be solved directly without the critically loaded assumption. 
It is straightforward to verify that \cref{Ex3.3B} becomes
\begin{equation*}
\begin{aligned}
\Bar{b}^n(x) \,=\, -\Bar{\mu}\bigl((x+ n^{-\beta}x^n_*)\wedge n^{1-\beta} 
- n^{-\beta}x^n_*\wedge n^{1-\beta}\bigr)
- \Bar{\gamma}\bigl( (x + n^{-\beta}x^n_* - n^{1-\beta})^+
-n^{-\beta}(x^n_*-n)^+\bigr)\,.
\end{aligned}
\end{equation*}
Repeating the procedure as above, we establish \cref{C2.2}.
\end{example}

\begin{example}[Markov-modulated $M/PH/n + M$ queues]\label{Ex3.4}
We assume that all customers start service in phase-$1$, and
there are $d$ phases. 
Given $J^n = k$,
the probability getting phase-$j$ after 
finishing service in phase-$i$ is denoted by $p_{ij}(k)$.
Let $X_1^n$ denote the total number of customers including in service and queue 
in phase-$1$, and
$X_i^n$, for $i\neq 1$, denotes the number of customers in service in phase-$i$. 
We refer reader to \cite{DHT10} for a detailed description of the model without
Markov modulation, and to \cite{Zhu-91} for an application of Markov-modulated
phase-type distributions in queueing. 
Then, \cref{E-Lambda} becomes 
\begin{equation*}
\begin{cases}
\Xi^n_1(x,k) \,=\, n\lambda(k)  
- \mu_1(k)\bigl(x_1 - (\langle e, x \rangle - n)^+\bigr)
- \gamma(k)\bigl(\langle e,x \rangle - n\bigr)^+ \\
\Xi^n_i(x,k) \,=\, -\mu_i(k)x_i 
+ \sum_{j\neq i,j\neq 1}p_{ji}(k)\mu_j(k)x_j 
+ p_{1i}(k)\mu_1(k)\bigl(x_1 - (\langle e, x \rangle - n)^+\bigr)
\quad \text{for\ } i\neq 1\,,
\end{cases}
\end{equation*}
and \cref{EA2.2A} becomes
$$
\begin{cases}
n \Bar\lambda  
- \Bar{\mu}_1\bigl(x^n_{*,1} - (\langle e, x^n_* \rangle - n)^+\bigr)
- \Bar{\gamma}\bigl(\langle e,x^n_* \rangle - n\bigr)^+ \,=\, 0\,,
\\
 - \Bar{\mu}_ix^n_{*,i} 
+ \sum_{j\neq i,j\neq 1}\Bar{p}_{ji}\Bar{\mu}_jx^n_{*,j}
+ \Bar{p}_{1i}\Bar{\mu}_1\bigl(x^n_{*,1} - (\langle e, x^n_* \rangle - n)^+\bigr)\,=\, 0
\quad \text{for\ } i\neq 1\,,
\end{cases}
$$
where $\Bar{\gamma} = \sum_{k\in\cK}\pi_k\gamma(k)$, 
and $\Bar{p}_{ij} = \sum_{k\in\cK}\pi_k p_{ij}(k)$.
Assume that $\Bar{\lambda} = 1$. 
Note that $\{\Xi^n_i\colon i\in\cI\}$ are piecewise linear functions
in their first argument.
It is straightforward to verify that \cref{H1} and \cref{A2.2} are satisfied. 
We get $x^n_*= n\rho$,
where 
$$
\rho \,\df\, \frac{\Bar{M}^{-1}e_1}{e\transp \Bar{M}^{-1}e_1}\,, \quad \text{and} \quad
\Bar{M} \,\df\, (I - \Bar{P}\transp)\diag(\Bar{\mu})\,,
$$
with the identity matrix $I$ and $\Bar{P}\df [\Bar{p}_{ij}]$.
The coefficients in \cref{E-sde} satisfy
\begin{align*}
\Bar{b}^n(x) &\,=\, - \Bar{M}x 
+ \bigl(\Bar{M} - \Bar{\gamma} I\bigr)e_1\langle e,x \rangle^+\,,\\
\Bar{a}^n_{ii}(0) &\,=\, \begin{cases}
n^{1-2\beta}\bigl(1 + \Bar{\mu}_1\rho_1\bigr)\,,\quad & \text{if } i=1\,, \\
n^{1-2\beta}\Bigl(\sum_{j\neq i,j\neq 1}\Bar{p}_{ji}\Bar{\mu}_j\rho_j + \Bar{\mu}_i\rho_i 
+ \Bar{\mu}_1\rho_1\Bar{p}_{1i}\Bigr)\,, \quad & \text{if } i\neq1\,,
\end{cases}
\intertext{and}
\Bar{a}^n_{ij}(0) &\,=\, 
n^{1-2\beta}\bigl(\Bar{p}_{ij}\Bar{\mu}_i\rho_i + \Bar{p}_{ji}\Bar{\mu}_j\rho_j\bigr)
\,, \quad i\neq j\,.
\end{align*}
By \cite[Theorem 3.5]{APS19} (see also \cite[Theorem 3]{Dieker-Gao}),
there exists a function $\widetilde{\sV}$ satisfying the assumption in \cref{C2.2}.
In analogy to \cite[Theorem 3.5]{APS19}, 
we can show that there exists a function 
$\Lyap(x) = \langle x, Rx \rangle^{\nicefrac{m}{2}}$, 
for $m\ge 2$ and some positive definite matrix $R$, satisfying 
the conditions in \cref{T2.1}.
\end{example}

\section{Proofs of \texorpdfstring{\cref{T2.1,C2.1}}{}}\label{S4}

We let $\cT=[\cT_{k\ell}]_{k,\ell\in\cK}$ denote a left pseudo-inverse of
the transition matrix $\cQ$.
That is,
\begin{equation}\label{E-cT}
\cQ\cT y \,=\, y\quad\text{for all\ } y\in\RR^{k_0}
\text{\ such that\ } \sum_{k\in\cK} \pi_k y_k=0\,.
\end{equation}
We also need the following definition.

\begin{definition}\label{D4.1}
Recall \cref{E-cL} and \cref{D2.2}.
Let $\Breve\cL^n_k\df\overline\cL^n-\cL^n_k$.
This operator takes the form
\begin{equation*}
\Breve\cL^n_kf(\Hat{x},k) \,\df\, \sum_{z\in\sZ^n}
\Breve{r}^n_k(n^{\beta}\Hat{x} + x^n_*,z)
\bigl(f(\Hat{x}+n^{-\beta}z,k)-f(\Hat{x},k)\bigr)\,,
\quad (\Hat{x},k) \in \widehat{\fX}^n\times\cK\,,
\end{equation*}
for $f\in C_b(\RR^d\times\cK)$,
where
$$\Breve{r}^n_k(x,z)\df 
\Bar{r}^n(x,z) - \Tilde{r}^n_k(x,z)\,,
\qquad (x,k)\in {\fX}^n\times\cK\,.$$ 
\end{definition}

\begin{proof}[Proof of \cref{T2.1}]
Let
\begin{equation}\label{E-cTA}
\widetilde{\Lyap}^n(\Hat{x},k) \,\df\, 
\frac{1}{n^{\alpha}}\sum_{\ell\in\cK}\cT_{k\ell}\,\Breve\cL^n_{\ell}\Lyap^n(\Hat{x})\,,
\quad (\Hat{x},k)\in\widehat{\fX}^n\times\cK\,.
\end{equation}
Then,
\begin{equation}\label{E-cTB}
\cQ^n\widetilde{\Lyap}^n(\Hat{x},k) \,=\,
\Breve\cL^n_k\Lyap^n(\Hat{x}) \qquad\forall\,(\Hat{x},k)\in\widehat{\fX}^n\times\cK\,,
\end{equation}
by \cref{E-cT}.

We define 
\begin{equation}\label{PT2.1A}
\widehat{\Lyap}^n(\Hat{x},k) 
\df \Lyap^n(\Hat{x}) + \widetilde{\Lyap}^n(\Hat{x},k)\,,\quad
(\Hat{x},k)\in\widehat{\fX}^n\times\cK\,.
\end{equation} 
By \cref{H1}\,(c) and (d), we have
\begin{equation}\label{PT2.1B}
\Tilde{r}^n_k(n^{\beta}\Hat{x} + x^n_*,z) 
\,\le\, C_0 (n^{\nicefrac{\alpha}{2}}+n^{\beta}\abs{\Hat{x}} + n)
\qquad \forall\,(\Hat{x},k)\in\widehat{\fX}^n\times\cK\,, 
\ \forall\,z\in\sZ^n\,, \ \forall\,n\in\NN\,, 
\end{equation}
We choose $N_1$ large enough so that $m_0 \le \varepsilon_0 N_1^\beta$,
with $m_0$ as defined in \cref{H1}\,(a).
By \cref{H1}\,(a) and (b), \cref{ET2.1A,PT2.1B}, we have
\begin{equation}\label{PT2.1C}
\babs{\Breve\cL^n_{k} \Lyap^n(\Hat{x})}
\,\le\, N_0 C_0 (n^{\nicefrac{\alpha}{2}}+n^{\beta}\abs{\Hat{x}} + n)\, C m_0
\frac{1+\Lyap^n(\Hat{x})}{n^\beta (1+\abs{\Hat{x}})}
\end{equation}
for all $n\ge N_1$.
Therefore, 
since $\alpha + \beta - 1 \ge \nicefrac{\alpha}{2}$ for $\alpha > 0$,
it follows by \cref{PT2.1B,PT2.1A,PT2.1C} that there exists $n_1\in\NN$,
$n_1\ge N_1$, such that 
\cref{ET2.1C} holds.

Recall the definitions in \cref{E-HcLn,E-cL,E-Lavg}. 
We have  
\begin{equation*}
\overline{\cL}^n\Lyap^n(\Hat{x})
\,=\, \cL_k^n\Lyap^n(\Hat{x}) + \Breve\cL^n_k\Lyap^n(\Hat{x})
\,=\, \cL_k^n\Lyap^n(\Hat{x}) + \cQ^n  \widetilde{\Lyap}^n(\Hat{x},k)
\end{equation*}
by \cref{E-cTB}.
Therefore, since $\cQ^n \Lyap^n(\Hat{x}) = 0$, we obtain
\begin{equation}\label{PT2.1D}
\begin{aligned}
\widehat{\cL}^n \widehat{\Lyap}^n(\Hat{x},k) &\,=\,
\cL_k^n\Lyap^n(\Hat{x}) + \cL^n_k \widetilde{\Lyap}^n(\Hat{x},k)
+ \cQ^n  \widetilde{\Lyap}^n(\Hat{x},k) \\
&\,=\, \overline{\cL}^n\Lyap^n(\Hat{x})
+ \cL^n_k \widetilde{\Lyap}^n(\Hat{x},k)
\qquad \forall\,(\Hat{x},k)\in\widehat{\fX}^{n}\times\cK\,.
\end{aligned}
\end{equation}

We define the function 
$$
G^n_k(\Hat{x},z) \,\df\, 
\Breve{r}^n_k\bigl(n^{\beta}\Hat{x} + x^n_{*},z\bigr)
\bigl(\Lyap^n(\Hat{x}+n^{-\beta}z) - \Lyap^n(\Hat{x})\bigr)\,.
$$
It is straightforward to verify, using \cref{E-cTA}, that
\begin{equation}\label{PT2.1E}
\begin{aligned}
\cL^n_k \widetilde{\Lyap}^n(\Hat{x},k) 
\,=&\, \sum_{h\in\sZ^n}
\Tilde{r}^n_k(n^{\beta}\Hat{x} + {x}^n_*,h)
\bigl(\widetilde{\Lyap}^n(\Hat{x} + n^{-\beta}h,k)
- \widetilde{\Lyap}^n(\Hat{x},k)\bigr)\\
\,=&\, \frac{1}{n^{\alpha}} \sum_{h,z\in\sZ^n}
\Tilde{r}^n_{k}(n^{\beta}\Hat{x} + {x}^n_*,h) \sum_{\ell\in\cK}\cT_{k\ell}
\bigl(G^n_{\ell}(\Hat{x} + n^{-\beta}h,z) - G^n_{\ell}(\Hat{x},z)\bigr)\,.
\end{aligned}
\end{equation}
On the other hand, it follows by \cref{H1}\,(c), and a triangle
inequality, that 
\begin{equation}\label{PT2.1F}
\begin{aligned}
\abs{G^n_k(\Hat{x} + n^{-\beta}h,z) - G^n_k(\Hat{x},z)} 
&\,\le\,2 C_0(n^{\nicefrac{\alpha}{2}}+\abs{h})\,
\babs{\Lyap^n(\Hat{x}+n^{-\beta}z) - \Lyap^n(\Hat{x})} \\ 
&\mspace{50mu}+  
\babs{\Breve{r}^n_k(n^{\beta}\Hat{x} + x^n_{*} + h,z)}
\,\babs{\Lyap^n(\Hat{x} + n^{-\beta}z + n^{-\beta}h) \\
&\mspace{100mu}
- \Lyap^n(\Hat{x} + n^{-\beta}h) - \Lyap^n(\Hat{x}+ n^{-\beta}z) 
+ \Lyap^n(\Hat{x})}
\end{aligned}
\end{equation}
for all $h,z\in\sZ^n$.
As in \cref{PT2.1B}\,, we have
\begin{equation}\label{PT2.1G}
\abs{\Breve{r}^n_k(n^{\beta}\Hat{x} + x^n_*+h,z)}
\,\le\, C_0 (n^{\nicefrac{\alpha}{2}}+n^{\beta}\abs{\Hat{x}} + \abs{h} + n)
\quad \forall\,(\Hat{x},k)\in\widehat{\fX}^n\times\cK\,, 
\ \forall\,h,z\in\sZ^n\,,\ \forall\,n\in\NN\,.
\end{equation}
By \cref{ET2.1A} and \cref{H1}\,(a), we have
\begin{equation}\label{PT2.1H}
\begin{aligned}
\babs{\Lyap^n(\Hat{x}+n^{-\beta}z) - \Lyap^n(\Hat{x})}
&\,\le\, C m_0
\frac{1+\Lyap^n(\Hat{x})}{n^\beta (1+\abs{\Hat{x}})}\,,\\[3pt]
\babs{\Lyap^n(\Hat{x} + n^{-\beta}z + n^{-\beta}h)
- \Lyap^n(\Hat{x} + n^{-\beta}h) - \Lyap^n(\Hat{x}+ n^{-\beta}z) + \Lyap^n(\Hat{x})}
&\,\le\, C m_0^2
\frac{1+\Lyap^n(\Hat{x})}{n^{2\beta} (1+\abs{\Hat{x}}^2)}
\end{aligned}
\end{equation}
for all $h,z\in B_{m_0}$, $\Hat{x}\in\widehat\fX^n$, and $n\in\NN$.
Hence, using \cref{PT2.1E} together with 
the estimates in \cref{PT2.1B,PT2.1F,PT2.1G,PT2.1H},
\cref{H1}\,(a) and (b), we obtain
\begin{equation}\label{PT2.1I}
\begin{aligned}
\cL^n_k \widetilde{\Lyap}^n(\Hat{x},k) &\,\le\,
N_0C_0C m_0\sum_{k,k'\in\cK}\abs{\cT_{k\ell}}\biggl(
2(n^{\nicefrac{\alpha}{2}} + m_0)(n^{\nicefrac{\alpha}{2}}+n^{\beta}\abs{\Hat{x}} + n)
\frac{1+\Lyap^n(\Hat{x})}{n^{\alpha+\beta} (1+\abs{\Hat{x}})}\\
&\mspace{40mu}+
N_0C_0m_0 (n^{\nicefrac{\alpha}{2}}+n^{\beta}\abs{\Hat{x}} + n)
(n^{\nicefrac{\alpha}{2}}+n^{\beta}\abs{\Hat{x}} + m_0 + n)
\frac{1+\Lyap^n(\Hat{x})}{n^{\alpha+2\beta} (1+\abs{\Hat{x}}^2)}\biggr)\,.
\end{aligned}
\end{equation}
Using the property
$\beta=\max\{\nicefrac{1}{2},1 - \nicefrac{\alpha}{2}\}$, we deduce from \cref{PT2.1I}
that for any $\epsilon>0$ there exists some constant $C_\circ(\epsilon)$ such that
\begin{equation}\label{PT2.1J}
\cL^n_k \widetilde{\Lyap}^n(\Hat{x},k)
\,\le\, C_\circ(\epsilon) + \epsilon \Lyap^n(\Hat{x})
\qquad\forall\,(\Hat{x},k)\in\widehat{\fX}^n\times\cK\,,\ \forall n\in\NN\,.
\end{equation}
Therefore, choosing $\epsilon=\frac{1}{2}\overline{C}_2$,
and using \cref{ET2.1B,ET2.1C,PT2.1D,PT2.1J}, we obtain
\begin{equation*}
\widehat{\cL}^n \widehat{\Lyap}^n(\Hat{x},k) \,\le\, \overline{C}_1
+ C_\circ\bigl(\nicefrac{\overline{C}_2}{2})
+\frac{1}{6}\overline{C}_2 - \frac{1}{3}\overline{C}_2
\widehat{\Lyap}^n(\Hat{x},k) \qquad \forall\,(\Hat{x},k)\in\widehat{\fX}^n\times\cK\,,
\ \forall\, n>n_1\,.
\end{equation*}
This completes the proof.
\end{proof}

\smallskip

\begin{proof}[Proof of \cref{C2.1}]
Recall $\cG^n_k$ in \cref{EC2.1A}, 
and let $\Breve\cG^n_k \df \cG^n_k - \cL^n_k$.
Then, $\Breve\cG^n_k$ takes the form
$$
\Breve\cG^n_k f(\Hat{x},k) \,=\, \sum_{z\in\sZ^n} \Breve\psi^n_k(n^\beta\Hat{x} + x^n_*,z)
\bigl(f(\Hat{x}+n^{-\beta}z,k)-f(\Hat{x},k)\bigr)\,,
\quad (\Hat{x},k) \in \widehat{\fX}^n\times\cK\,,
$$
for $f\in C_b(\RR^d\times\cK)$,
where
$$\Breve{\psi}^n_k(x,z)\df 
\Bar{\psi}^n(x,z) - \psi^n_k(x,z)\,, \quad k\in\cK\,,
\quad (x,z)\in\fX^n\times\sZ^n\,.$$ 
Compare it to \cref{D4.1}. In analogy to \cref{PT2.1D},
we let
\begin{align*}
\cQ^n\widetilde{\Lyap}^n(\Hat{x},k) &\,=\, \Breve\cG^n_k\Lyap^n(\Hat{x})\,, \qquad
 (\Hat{x},k)\in\widehat{\fX}^{n}\times\cK\,, \\
\intertext{and get}
\widehat\cL^n \widehat{\Lyap}^n(\Hat{x},k) &\,=\, 
\cG^n_k\Lyap^n(\Hat{x})
+ \cL^n_k \widetilde{\Lyap}^n(\Hat{x},k)\,, 
\qquad  (\Hat{x},k)\in\widehat{\fX}^{n}\times\cK\,.
\end{align*}
In obtaining an estimate for $\cL^n_k \widetilde{\Lyap}^n(\Hat{x},k)$,
the proof is the same to that of \cref{T2.1} 
by replacing $\Breve r^n$ with $\Breve\psi^n$, and 
using \cref{A2.1B,A2.1C}.
Applying \cref{A2.1B,A2.1C} again, 
we may show \cref{ET2.1C}.
Then, the claim in \cref{ET2.1D} follows by \cref{EC2.1B}.
\end{proof}

\section{Proofs of \texorpdfstring{\cref{T2.2,C2.2}}{}}\label{S5}

We need to introduce some additional notation to facilitate
the proofs.
Recall the definitions of $\Bar{b}^n$, $\Bar{a}^n$, and $\Bar{\Gamma}^n$
 in \cref{E-sde,E-Sigma,E-bargamma},
respectively.
For $f\in C^2(\Rd)$ and $n\in\NN$,  let
\begin{equation}\label{E-Tg1}
\begin{aligned} 
\Tilde{g}^n_1[f](x,k) 
&\,\df\, \sum_{i\in\cI}\bigl(\Bar{b}^n_i(x) - 
\bigl(\widehat{\Xi}^n_i(x,k) - \widehat{\Xi}^n_i(0,k) \bigr)\bigr)\partial_i f(x) \\ 
&\mspace{180mu} 
+ \frac{1}{2}\sum_{i,j\in\cI}\bigl(\Bar{a}^n_{ij}(x)
- \Bar{\Gamma}^n_{ij}(x,k)\bigr)\partial_{ij}f(x)\,,
\end{aligned}
\end{equation}
and
\begin{equation}\label{E-Tg2}
\begin{aligned}
\Tilde{g}^n_2[f](x,k) 
&\,\df\, \frac{1}{n^{\alpha + 2\beta}}\sum_{z}\sum_{h\in\cK} 
\biggl(\sum_{l\in\cK}\pi_l\xi^n_z
(n^{\beta}x + {x}^n_{*},l)\Upsilon_{lh} \\
&\mspace{160mu} 
- \xi^n_z(n^{\beta}x + {x}^n_{*},k)\Upsilon_{kh}\biggr) 
\sum_{j\in\cI}\Xi^n_j({x}^n_{*},h)
\sum_{i\in\cI}z_i\partial_{ij} f(x)\,.
\end{aligned}
\end{equation}
It follows by the identity
$$
\sum_{k\in\cK}\Bigl(\sum_{l\in\cK}\pi_l\xi^n_z
(n^{\beta}x + {x}^n_{*},l)\Upsilon_{lh} 
- \xi^n_z(n^{\beta}x + {x}^n_{*},k)\Upsilon_{kh}\Bigr)\,\equiv\,0\,,
$$
that $\sum_{k\in\cK}\pi_k \Tilde{g}^n_2[f](x,k) \,=\, 0$.
It is clear that
$\sum_{k\in\cK}\pi_k \Tilde{g}^n_1[f](x,k) = 0$.
Recall the mapping $\cT$ in \cref{E-cT}. We define
\begin{equation}\label{E-g1g2}
g^n_{i}[f](x,k) \,\df\,  \frac{1}{n^{\alpha}}
\sum_{\ell\in\cK}\cT_{k\ell}\, \Tilde{g}^n_i[f](x,\ell)\,,\quad i=1,2\,,
\end{equation}
and thus
\begin{equation}\label{E-g1g2B}
\cQ^ng^n_i[f](x,k) \,=\, \Tilde{g}^n_i[f](x,k)\,,
\quad i=1,2\,.
\end{equation}

For $f\in C^2(\Rd)$ and $n\in\NN$, let 
\begin{equation}\label{E-g3}
g^n_{3}[f](x,k)\,\df\, \frac{1}{n^{\alpha+\beta}}
\sum_{h\in\cK}\sum_{j\in\cI}\Xi^n_j({x}^n_{*},h)
\Upsilon_{kh}\,\partial_jf(x)\,.
\end{equation}
Note that the function $g^n_3[f]$ corresponds to the covariance of the 
background Markov process $J^n$.
We let  $g^n[f]$ denote the sum of the above functions, that is,
\begin{equation}\label{E-g}
g^n[f](x,k) \,\df\, g^n_{1}[f](x,k)
+ g^n_{2}[f](x,k) + g^n_{3}[f](x,k)\,,
\quad  (x,k)\in\RR^d\times\cK\,.
\end{equation}

To keep the algebraic expressions in the proofs
 manageable, we adopt the notation introduced in the following definition.

\begin{definition}\label{D5.1}
We define the operators $[\cD^n_z]^{0}$ and $[\cD^n_z]^{1}_j$, $j\in\cI$, by
\begin{equation*}
\begin{aligned}
[\cD^n_z]^{0}f(x) &\,\df\, f(x+n^{-\beta}z) - f(x)
- n^{-\beta}\sum_{i\in\cI}z_i\partial_if(x)
- n^{-2\beta}\sum_{i,j\in\cI}z_iz_j\partial_{ij}f(x)\,,\\
[\cD^n_z]^{1}_jf(x) &\,\df\, \partial_jf(x+n^{-\beta}z) - \partial_jf(x)
- n^{-\beta}\sum_{i\in\cI}z_i\partial_{ij} f(x)\,, 
\end{aligned}
\end{equation*}
for  $f\in C^2(\RR^d)$ and $z\in\sZ^n$.
In addition, we define
\begin{equation*}
\begin{aligned}
\cR^n_1[f](\Hat{x},k) &\,\df\, \sum_{z}^{}\xi^n_{z}(n^{\beta}\Hat{x}
+ {x}^n_{*},k)\,[\cD^n_z]^{0}f(\Hat{x})\,,\\
\cR^n_2[f](\Hat{x}) &\,\df\, \frac{1}{2}\sum_{i,j\in\cI}\sum_{k\in\cK}
\pi_k\bigl(\Bar{\Gamma}^n_{ij}(\Hat{x},k) - \Bar{\Gamma}^n_{ij}(0,k)\bigr)
\partial_{ij}f(\Hat{x})\,,\\
\cR^n_3[f](\Hat{x},k) &\,\df\, \frac{1}{n^{\alpha+2\beta}}\sum_{i,j\in\cI}
\sum_{h\in\cK}\sum_{l\in\cK}
\bigl(\Xi^n_i({x}^n_* + n^{\beta}\Hat{x},l) - \Xi^n_i({x}^n_*,l)\bigr)
\Xi^n_j({x}^n_{*},h)\pi_l\Upsilon_{lh}\partial_{ij}f(\Hat{x})\,,\\
\cR^n_4[f](\Hat{x},k) &\,\df\, \frac{1}{n^{\alpha+\beta}}\sum_{z}\sum_{h\in\cK}
\xi^n_z(n^{\beta}\Hat{x} + {x}^n_{*},k)\Upsilon_{kh}
\sum_{j\in\cI}\Xi^n_j({x}^n_{*},h)[\cD^n_z]^{1}_jf(\Hat{x})\,,\\
\cR^n_5[f](\Hat{x},k) &\,\df\, \cL^n_k\,g^n_{1}[f](\Hat{x},k)\,,\\[3pt]
\cR^n_6[f](\Hat{x},k) &\,\df\, \cL^n_{k}\,g^n_{2}[f](\Hat{x},k)\,.
\end{aligned}
\end{equation*}
\end{definition}

The following lemma establishes a useful identity involving 
the generator of $(\widehat{X}^n,J^n)$ in \cref{E-HcLn}
and that of $\widehat{Y}^n$ in \cref{E-sAn} and the operators
$\cR^n_i$ in \cref{D5.1}.

\begin{lemma}\label{L5.1}
Under \cref{A2.2}\,\ttup{ii}, we have
\begin{equation}\label{EL5.1A}
\widehat\cL^n f(\Hat{x}) +
\widehat\cL^n g^n[f](\Hat{x},k) 
\,=\, \sA^n f(\Hat{x}) + \sum_{i=1}^6 \cR^n_i[f](\Hat{x},k)\,,
\quad (\Hat{x},k)\in\widehat{\fX}^n\times\cK\,,\ f\in C^2(\RR^d)\,.
\end{equation}
\end{lemma}

\begin{proof}
By \cref{E-HcLn} we have
\begin{equation}\label{PL5.1A}
\widehat\cL^n g^n[f](\Hat{x},k) \,=\,
\sum_{i=1}^3\Bigl(\cL^n_kg^n_{i}[f](\Hat{x},k)
+ \cQ^ng^n_{i}[f](\Hat{x},k)\Bigr)\,,
\end{equation}
and $\widehat\cL^nf(\Hat{x}) = \cL^n_k f(\Hat{x})$ for any $f\in C^2(\RR^d)$.

We first show that 
\begin{equation}\label{PL5.1B}
\begin{aligned}
&\cL^n_kf(\Hat{x}) + \cQ^ng^n_{1}[f](\Hat{x},k)
+ \cQ^ng^n_{3}[f](\Hat{x},k) \\
&\mspace{100mu}\,=\, \sum_{i\in\cI}\Bar{b}^n_i(\Hat{x}) \partial_if(\Hat{x}) 
+ \frac{1}{2}\sum_{i,j\in\cI}\Bar{a}^n_{ij}\partial_{ij}f(\Hat{x})
+ \cR^n_1[f](\Hat{x},k) + \cR^n_2[f](\Hat{x})\,.
\end{aligned}
\end{equation}
Using \cref{E-cL,E-g3}, we obtain
\begin{equation}\label{PL5.1C}
\cQ^ng^n_{3}[f](\Hat{x},k) \,=\, 
\sum_{h\in\cK}\sum_{\ell\in\cK} q_{k\ell}\Upsilon_{\ell h}
\sum_{j\in\cI}\frac{\Xi^n_j({x}^n_*,h)}{n^\beta}\partial_jf(\Hat{x})\,.
\end{equation}
Since $Q\Upsilon = \Pi - I$, 
where $I$ denotes the identity matrix, it follows by \cref{EA2.2A} that
\begin{equation}\label{PL5.1D}
\sum_{h\in\cK}\sum_{\ell\in\cK}q_{k\ell}\Upsilon_{\ell h}\Xi^n_j({x}^n_{*},h)
\,=\, \sum_{h\in\cK}\pi_h\Xi^n_j({x}^n_*,h) - \Xi^n_j({x}^n_*,k)
\,=\, - \Xi^n_j({x}^n_*,k)\,,
\end{equation}
where in the second equality we use \cref{A2.2}\,(ii).
Thus, by \cref{PL5.1C,PL5.1D}, we have
\begin{equation}\label{PL5.1E}
\cQ^ng^n_{3}[f](\Hat{x},k) \,=\,
\sum_{j\in\cI} -\frac{\Xi^n_j({x}^n_*,k)}{n^{\beta}}\,\partial_jf(\Hat{x}) 
\,=\,\sum_{j\in\cI} -\widehat{\Xi}^n_j(0,k)\,\partial_jf(\Hat{x})\,.
\end{equation}
By \cref{E-cL} and a standard identity, we obtain
\begin{equation}\label{PL5.1F}
\begin{aligned}
\cL^n_kf(\Hat{x}) &\,=\, \sum_{z\in\sZ^n} 
\xi^n_z(n^{\beta}\Hat{x} 
+ x^n_*,k)\biggl(\sum_{i\in\cI}n^{-\beta}z_i\partial_i f(\Hat{x})  
+ \sum_{i,j\in\cI}n^{-2\beta}z_iz_j\partial_{ij}f(\Hat{x})
+ [\cD^n_z]^0 f(\Hat{x})\biggr) 
\\
&\,=\, \sum_{i\in\cI}\widehat{\Xi}^n_i(\Hat{x},k)\partial_if(\Hat{x}) + 
\sum_{i,j\in\cI} \Bar{\Gamma}^n_{ij}(\Hat{x},k) \partial_{ij}f(\Hat{x})
+ \cR^n_1[f](\Hat{x},k)\,.
\end{aligned}
\end{equation}
Thus \cref{PL5.1B} follows from \cref{E-Tg1,E-g1g2B,PL5.1E,PL5.1F}.

Next, we show that
\begin{equation}\label{PL5.1G}
\cL^n_k\,g^n_{3}[f](\Hat{x},k) + \cQ^ng^n_{2}[f](\Hat{x},k) \,=\, 
\frac{1}{2}\sum_{i,j\in\cI}\theta^n_{ij}\partial_{ij}f(\Hat{x}) 
+ \cR^n_3[f](\Hat{x},k) + \cR^n_4[f](\Hat{x},k)\,. 
\end{equation}
We have
$$
\cL^n_k\,g^n_3[f](\Hat{x},k) \,=\, \frac{1}{n^{\alpha + \beta}}
\sum_{z}\xi^n_{z}(n^{\beta}\Hat{x} + {x}^n_{*},k)
\sum_{h\in\cK}\sum_{j\in\cI}\Xi^n_j({x}^n_{*},h)\Upsilon_{kh}\,
\bigl(\partial_jf(\Hat{x} + n^{-\beta}z) - \partial_jf(\Hat{x})\bigr)
$$
by \cref{E-cL}.
It is clear that
$$
\partial_jf(\Hat{x}+n^{-\beta}z) - \partial_jf(\Hat{x}) 
\,=\, n^{-\beta}\sum_{i\in\cI}z_i\partial_{ij} f(\Hat{x}) +
[\cD^n_z]^{1}_jf(\Hat{x})\,,
$$
and
$$
\sum_{z}z_i\xi^n_{z}(n^{\beta}\Hat{x} + {x}^n_{*},k) \,=\, 
\Xi^n_i({x}^n_*,k) 
+ \bigl(\Xi^n_i({x}^n_* + n^{\beta}\Hat{x},k) - \Xi^n_i({x}^n_*,k)\bigr)\,.
$$
Therefore, \cref{PL5.1G} follows by combining these identities with \cref{E-Tg2,E-g1g2B}.

Hence, we obtain \cref{EL5.1A} by
adding \cref{PL5.1A,PL5.1B,PL5.1G}, and using the definitions
of $\cR^n_i[f]$ for $i=5,6$.
This completes the proof.
\end{proof}

The estimates for $\cR^n_i$, $i=1,2,3,4$, are discussed in \cref{L5.4}.
The following lemma provides needed estimates for $\cR^n_5$ and $\cR^n_6$.

\begin{lemma}\label{L5.2}
Under \cref{A2.2}\,\ttup{i}--\ttup{iii}, there exists some positive constant
$C$ such that
\begin{equation}\label{EL5.2A}
\begin{aligned}
\babs{\cR^n_5[f](\Hat{x},k)} &\,\le\,  
C \biggl[\biggl(\frac{1}{n^{\alpha}}\abs{\Hat{x}}
+ \frac{1}{n^{\alpha+\beta -1}} \biggr) \abs{\grad f(\Hat{x})} 
+ \biggl(\frac{1}{n^{\alpha+\beta}}\abs{\Hat{x}} + \frac{1}{n^{\alpha+2\beta-1}} \biggr)
\babs{\grad^2 f(\Hat{x})}
\\
&\quad + \biggl(\frac{1}{n^{\alpha - \beta}}\abs{\Hat{x}}^2
+ \frac{1}{n^{\alpha-1}}\abs{\Hat{x}}\biggr)
\,\max_{z\in\sZ^n}\bigl\{
\abs{\grad f(\Hat{x} + n^{-\beta}z) - \grad f(\Hat{x})}\bigr\}\\
&\quad + \biggl(\frac{1}{n^{\alpha}}\abs{\Hat{x}}^2
+ \frac{1}{n^{\alpha+\beta-1}}\abs{\Hat{x}} 
+ \frac{1}{n^{\alpha + 2\beta - 2}}\biggr)
\,\max_{z\in\sZ^n}
\babs{ \grad^2 f(\Hat{x}+n^{-\beta}z) - \grad^2 f(\Hat{x})}\bigr\}
\biggr]\,,
\end{aligned}
\end{equation}
and 
\begin{equation}\label{EL5.2B}
\begin{aligned}
\babs{\cR^n_6[f](\Hat{x},k)} &\,\le\, C\biggl[
\biggl(\frac{1}{n^{2\alpha+\beta -1}}\abs{\Hat{x}}
+ \frac{1}{n^{2\alpha+2\beta-2}}\biggr)\,
\babs{\grad^2 f(\Hat{x})} \\
&\mspace{40mu} + 
\biggl(\frac{1}{n^{2\alpha - 1}}\abs{\Hat{x}}^2
+ \frac{1}{n^{2\alpha + \beta -2}}\abs{\Hat{x}} \\
&\mspace{200mu}
+ \frac{1}{n^{2\alpha + 2\beta - 3}}\biggr)
\max_{z\in\sZ^n}\,
\babs{\grad^2 f(\Hat{x}+n^{-\beta}z) - \grad^2 f(\Hat{x})}\biggr]\,, 
\end{aligned}
\end{equation}
for any $(\Hat{x},k)\in\widehat{\fX}^n\times\cK$.
\end{lemma}

\begin{proof}
Recall the functions $g^n_1[f]$ and $g^n_{2}[f]$ in \cref{E-g1g2}.
It follows by \cref{EA2.2B,EA2.2C} that
\begin{align}\label{PL5.2A}
\abs{\xi^n_z(x, k)} \,\le\,
\widetilde{C}(\abs{x - x^n_*} + n)\,, 
\intertext{and}
\abs{\xi^n_z(x+x^n_*,k) - \xi^n_z(x^n_*,k)} \,\le\, \widetilde{C}\abs{x}\,,
\label{PL5.2B}
\end{align}
for $(x,k)\in\RR^d\times\cK$, $z\in\sZ^n$ and $n\in\NN$.
By \cref{A2.2} (i), and applying \cref{EA2.2A,PL5.2B}, 
it is straightforward to verify that
\begin{equation}\label{PL5.2C}
\Biggl| \sum_{k\in\cK} \pi_k \widehat\Xi^n(\Hat{x},k) \Biggr|
\,\le\, \widetilde{C}\widetilde{N}_0m_0\abs{\Hat{x}} 
\qquad \forall\,\Hat{x}\in\Rd\,.
\end{equation}
Thus, by \cref{PL5.2C,PL5.2B}, we have
\begin{equation}\label{PL5.2D}
\babs{\Bar{b}^n(\Hat{x}) - \bigl(\widehat{\Xi}^n(\Hat{x},k) 
- \widehat{\Xi}^n(0,k)\bigr)}
\,\le\, 2\widetilde{C}\widetilde{N}_0m_0\abs{\Hat{x}}
\qquad \forall\,(\Hat{x},k)\in\RR^d\times\cK\,.
\end{equation}
Applying \cref{PL5.2A}, we obtain
\begin{equation}\label{PL5.2E}
\babs{\Bar{a}^n(\Hat{x}) - \Bar{\Gamma}^n(\Hat{x},k)} 
\,\le\, 2\widetilde{C}\widetilde{N}_0m_0^2
\bigl(n^{-\beta}\abs{\Hat{x}} + n^{1- 2\beta}\bigr) 
\qquad \forall (\Hat{x},k)\in\RR^d\times\cK\,,
\end{equation}
and
\begin{equation}\label{PL5.2F}
\Biggl| \sum_{l\in\cK}\pi_l\xi^n_z
(n^{\beta}\Hat{x} + {x}^n_{*},l)\Upsilon_{lh}
- \xi^n_z(n^{\beta}\Hat{x} + {x}^n_{*},k)\Upsilon_{kh} \Biggr|
\,\le\, C_1\bigl(n^{\beta}\abs{\Hat{x}} + n\bigr) \qquad 
\forall\, \Hat{x}\in\RR^d\,,
\end{equation}
and all $k,h\in\cK$ and $z\in\sZ^n$, for some positive constant $C_1$.
We have 
\begin{equation}\label{PL5.2G}
\abs{\Xi^n(x^n_*,k)} \,\le\, \widetilde{C}\widetilde{N}_0m_0n
\qquad \forall\,k\in\cK\,,\ n\in\NN\,,
\end{equation}
by \cref{EA2.2C},
and 
\begin{equation}\label{PL5.2H}
\abs{\xi^n_z(n^{\beta}\Hat{x} + x^n_*,k)} \,\le\, 
\widetilde{C}(n^{\beta}\abs{\Hat{x}} +n) \qquad 
\forall\,(\Hat{x},k) \in\RR^d\times\cK\,,\ z\in\sZ^n\,,\ n\in\NN\,,
\end{equation}
by \cref{PL5.2A}.
From \cref{EA2.2B}, we obtain
\begin{equation}\label{PL5.2I}
\babs{\widehat{\Xi}^n(\Hat{x} + n^{-\beta}z,k) -  \widehat{\Xi}^n(\Hat{x},k)}
\,\le\, n^{-\beta}\widetilde{C}\widetilde{N}_0m_0^2\,,
\end{equation}
and 
\begin{equation}\label{PL5.2J}
\babs{\Bar{\Gamma}^n(\Hat{x} + n^{-\beta}z,k) - \Bar{\Gamma}^n(\Hat{x},k)}
\,\le\, n^{-2\beta}\widetilde{C}\widetilde{N}_0m_0^3\,,
\end{equation}
for $(\Hat{x},k)\in\RR^d\times\cK$, $z\in\sZ^n$, and $n\in\NN$.
Repeating similar calculations as in \cref{PT2.1F,PT2.1I},
and applying \cref{PL5.2D,PL5.2E,PL5.2H,PL5.2I,PL5.2J},
we have
\begin{equation*}
\begin{aligned}
\babs{\cR^n_5[f](\Hat{x}),k} &\,\le\, 
\widetilde{C}\widetilde{N}_0m_0
\sum_{k,\ell\in\cK}\abs{\cT_{k\ell}}\biggl( 
2\widetilde{C}\widetilde{N}_0m_0^2
\frac{(n^{\beta}\abs{\Hat{x}} +n)}{n^{\alpha+\beta}} 
\abs{\grad f(\Hat{x})}\\
&\mspace{20mu} + 2\widetilde{C}\widetilde{N}_0m_0\abs{\Hat{x}}
\frac{(n^{\beta}\abs{\Hat{x}} +n)}{n^{\alpha}}\max_{z\in\sZ^n}\bigl\{
\abs{\grad f(\Hat{x} + n^{-\beta}z) - \grad f(\Hat{x})}\bigr\}\\
&\mspace{20mu} + \widetilde{C}\widetilde{N}_0m_0^3
\frac{(n^{\beta}\abs{\Hat{x}} +n)}{n^{\alpha+2\beta}} \abs{\grad^2 f(\Hat{x})}\\
&\mspace{20mu} + 2\widetilde{C}\widetilde{N}_0m_0^2
\frac{\bigl(n^{-\beta}\abs{\Hat{x}} + n^{1- 2\beta}\bigr)(n^{\beta}\abs{\Hat{x}} +n)}
{n^{\alpha}}\max_{z\in\sZ^n}\bigl\{
\abs{\grad f(\Hat{x} + n^{-\beta}z) - \grad f(\Hat{x})}\bigr\}
\biggr)\,,
\end{aligned} 
\end{equation*}
which establishes \cref{EL5.2A}.
The estimate for $\cR^n_6$ in \cref{EL5.2B}  obtained in a similar manner
by applying \cref{PL5.2F,PL5.2G,EA2.2B,PL5.2H}.
This completes the proof.
\end{proof}
 
We borrow the following estimates for solutions
to the Poisson equation for the operator $\cA^n$ from
\cite[Theorem~4.1]{Gurvich14} and the discussion following this theorem.
Recall that  $\nu^n$ is the steady-state distribution of $\Hat{Y}^n$ in \cref{E-sde}.

\begin{lemma}\label{L5.3}
Grant  \cref{A2.2}, and fix a function $\sV$ in \cref{A2.3}.
Let $f\in C^{0,1}(\RR^d)$ be such that $\norm{f}_{C^{0,1}(\sB_x)} \le \sV(x)$
and $\nu^n(f) = 0$.
Then, the function $u^n_f\in C^2(\RR^d)$ defined by
\begin{equation*}
u^n_f(x) \,\df\, \int_0^{\infty} \Exp_x\bigl[f\bigl(\widehat{Y}^n(s)\bigr)\bigr]\,\D{s} 
\end{equation*}
is the unique  (up to an additive constant) solution to the Poisson equation
\begin{equation}\label{EL5.3A}
\sA^n u = -f\,.
\end{equation}
and satisfies
\begin{equation}\label{EL5.3B}
\abs{\grad u^n_f(x)} \,\in\, \order\bigl((1 + \abs{x})\sV(x)\bigr)\,,
\qquad \babs{\grad^2 u^n_f(x)} \,\in\, 
\order\bigl((1 + \abs{x}^2)\sV(x)\bigr)\,, 
\end{equation}
and
\begin{equation}\label{EL5.3C}
\bigl[u^n_f\bigr]_{2,1;B_{\frac{m_0}{\sqrt{n}}}(x)} \,\in\,
\order\bigl((1 + \abs{x}^3)\sV(x) \bigr)\,.
\end{equation}
\end{lemma}

In the following lemma, we consider the solution of the Poisson equation in \cref{EL5.3A},
and establish an estimate for the residual terms
$\cR^n_i[u^n_f]$, $i=1,\dotsc,6$, in \cref{EL5.1A} given in \cref{D5.1}.

\begin{lemma}\label{L5.4}
Grant \cref{A2.2}, and fix a function $\sV$ in \cref{A2.3}. 
Let $f$ and $u^n_f$ be  as in \cref{L5.3}.
Then, 
\begin{equation}\label{EL5.4A}
\sum_{j=1}^{6}\cR^n_j[u^n_f](\Hat{x},k) \,=\, 
\order\biggl(\frac{1}{n^{\nicefrac{\alpha}{2}\wedge\nicefrac{1}{2}}}\biggr)
\order\bigl((1 + \abs{\Hat{x}}^5)\sV(\Hat{x})\bigr) 
\qquad \forall\,(\Hat{x},k)\in\widehat{\fX}^n\times\cK\,.
\end{equation}
\end{lemma}

\begin{proof}
Note that
\begin{equation*}
[\cD^n_z]^{0}u^n_f(\Hat{x}) \,=\, 
n^{-2\beta}\sum_{i,j\in\cI}z_iz_j\partial_{ij}
\bigl(u^n_f(\Hat{x} + \varepsilon^n_{x,z}) - u^n_f(\Hat{x})\bigr)
\end{equation*}
for $\varepsilon^n_{\Hat{x},z}\in 
\prod_{i\in\cI}[\Hat{x}_i, \Hat{x}_i + n^{-\beta}z_i]$.
Applying \cref{PT2.1B,EL5.3C}, we obtain 
\begin{equation}\label{PL5.4A}
\cR^n_1[u^n_f](\Hat{x},k) \,=\, 
\frac{1}{n^{\beta}}\order\bigl((1 + \abs{\Hat{x}}^4)\sV(\Hat{x})\bigr)
\qquad \forall\, (\Hat{x},k)\in\widehat{\fX}^n\times\cK\,.
\end{equation}
By \cref{PL5.2B}, we have
$$
\abs{\Bar{\Gamma}^n_{ij}(\Hat{x},k) - \Bar{\Gamma}^n_{ij}(0,k)} \,\le\, 
\widetilde{C}\widetilde{N}_0m_0^2n^{-\beta}\abs{\Hat{x}} 
\qquad \forall\, (\Hat{x},k)\in\widehat{\fX}^n\times\cK\,,
$$
and thus it follows by \cref{EL5.3B} that
\begin{equation}\label{PL5.4B}
\cR^n_2[u^n_f](\Hat{x}) \,=\,
\frac{1}{n^{\beta}}\order\bigl((1 + \abs{\Hat{x}}^3)\sV(\Hat{x})\bigr)\,.
\end{equation} 
Applying \cref{D5.1}, \cref{PL5.2B,PL5.2G,EL5.3B},
we obtain
\begin{equation}\label{PL5.4C}
\cR^n_3[u^n_f](\Hat{x},k) \,=\, \frac{1}{n^{\alpha + \beta - 1}}
\order\bigl((1 + \abs{\Hat{x}}^3)\sV(\Hat{x})\bigr)\qquad \forall\,k\in\cK\,.
\end{equation}
Repeating the above procedure, and using \cref{D5.1}, 
\cref{PL5.2A,PL5.2G,EL5.3C}, we obtain
\begin{equation}\label{PL5.4D}
\cR^n_4[u^n_f](\Hat{x},k) \,=\,  \order\biggl(\frac{1}{n^{\alpha+3\beta-2}}\biggr)
\order\bigl((1 + \abs{\Hat{x}}^4)\sV(\Hat{x})\bigr) \qquad \forall\,k\in\cK \,.
\end{equation}
It follows by \cref{L5.2}, \cref{EL5.3B,EL5.3C} that
\begin{align}
\cR^n_5[u^n_f](\Hat{x},k) &\,=\, \order\biggl(\frac{1}{n^{\alpha + \beta- 1}}\biggr)
\order\bigl((1 + \abs{\Hat{x}}^5)\sV(\Hat{x})\bigr)\,,\label{PL5.4E} \\
\intertext{and }
\cR^n_6[u^n_f](\Hat{x},k) &\,=\, \order\biggl(\frac{1}{n^{2\alpha + 3\beta- 3}}\biggr)
\order\bigl((1 + \abs{\Hat{x}}^5)\sV(\Hat{x})\bigr)\,,\label{PL5.4F}
\end{align}
for all $k\in\cK$. On the other hand,
when $\alpha > 1$,
$\beta = \frac{1}{2}$, $\alpha + \beta - 1 \ge \beta$ and $2\alpha + 3\beta -3 \ge \beta$,  
and when $\alpha \le 1$,
$\alpha + \beta - 1 = 2\alpha + 3\beta - 3 = \nicefrac{\alpha}{2}$
and $\alpha + 3\beta - 2 = \beta$.
Then, by using 
\cref{PL5.4A,PL5.4B,PL5.4C,PL5.4D,PL5.4E,PL5.4F}, we have shown \cref{EL5.4A}.
This completes the proof.
\end{proof}

\smallskip
\begin{proof}[Proof of \cref{T2.2}]
Without loss of generality, we assume that $\nu^n(f) = 0$
(see \cite[Remark 3.2]{Gurvich14}).
Recall the function $g^n$ in \cref{E-g}.
Applying \cref{L5.1}, it follows that
\begin{equation}\label{PT2.2A}
\begin{aligned}
&\Exp_{\uppi^n}\Bigl[u^n_f\bigl(\widehat{X}^n(T)\bigr)
+ g^n[u^n_f]\bigl(\widehat{X}^n(T),J^n(T)\bigr)\Bigr]  \\
&\mspace{60mu}\,=\, 
\Exp_{\uppi^n}\Bigl[u^n_f(\widehat{X}^n(0))
+ g^n[u^n_f](\widehat{X}^n(0),J^n(0))\Bigr] \\
&\mspace{120mu}
+ \Exp_{\uppi^n}\biggl[\int_0^{T}\sA^n u^n_f\bigl(\widehat{X}^n(s)\bigr)\,\D{s}\biggr]
+\sum_{j=1}^{6}\Exp_{\uppi^n}\biggl[\int_0^{T}\cR^n_j[u_f^n]
\bigl(\widehat{X}^n(s),J^n(s)\bigr)\,\D{s}\biggr]\,.
\end{aligned}
\end{equation}
By \cref{L5.4}, we have
\begin{equation}\label{PT2.2B}
\begin{aligned}
\sum_{j=1}^{6}\Exp_{\uppi^n}\biggl[\int_0^{T}\cR^n_j[u_f^n]&
\bigl(\widehat{X}^n(s),J^n(s)\bigr)\,\D{s}\biggr]\\
&\,\le\, \order\biggl(\frac{1}{n^{\nicefrac{\alpha}{2}\wedge\nicefrac{1}{2}}}\biggr)
\Exp_{\uppi^n}\biggl[\int_0^{T}\Bigl(1 + \sV\bigl(\widehat{X}^n(s)\bigr)
\bigl(1 + \abs{\widehat{X}^n(s)}^5\bigr)\Bigr)\,\D{s}\biggr]\\ 
&\,=\, \order\biggl(\frac{1}{n^{\nicefrac{\alpha}{2}\wedge\nicefrac{1}{2}}}\biggr)
T \int_{\RR^d\times\cK}(1 +\sV(\Hat{x}))(1 + \abs{\Hat{x}})^5\uppi^n(\D{\Hat{x}},\D{k})\,.
\end{aligned}
\end{equation}
Applying \cref{E-g,PL5.2H,EL5.3B}, we obtain
\begin{equation}\label{PT2.2BB}
\abs{g^n(\Hat{x},k)} \,\le\, C_1\bigl(1 + (1 + \abs{\Hat{x}}^3)\sV(\Hat{x})\bigr)
\qquad \forall\,(\Hat{x},k)\in\widehat{\fX}^n\times\cK\,,
\end{equation}
for some positive constant $C_1$ and all large enough $n$.
Since $\babs{u^n_f}\in\order(\sV)$ by the claim in (22) of \cite{Gurvich14},
then it follows by \cref{PT2.2BB} that
\begin{equation}\label{PT2.2C}
\Babs{\Exp_{\uppi^n}\Bigl[u^n_f\bigl(\widehat{X}^n(T)\bigr)
+ g^n[u^n_f]\bigl(\widehat{X}^n(T),J^n(T)\bigr)\Bigr]} \,\le\, 
C_2\biggl( 1 + \int_{\RR^d\times\cK}\sV(\Hat{x})
(1 + \abs{\Hat{x}})^3\uppi^n(\D{\Hat{x}},\D{k})\biggr)
\end{equation}
for some positive constant $C_2$.
Therefore, dividing both sides of \cref{PT2.2A} by $T$ 
and taking $T\rightarrow\infty$,  and applying \cref{ET2.2A,EL5.3A,PT2.2B,PT2.2C},
we obtain
$$
\abs{\uppi^n(f)} \,=\, 
\order\biggl(\frac{1}{n^{\nicefrac{\alpha}{2}\wedge\nicefrac{1}{2}}}\biggr)\,.
$$
This completes the proof.
\end{proof}

\smallskip
\begin{proof}[Proof of \texorpdfstring{\cref{C2.2}}{}]
We claim that for some positive constants $C_1$, $\kappa_1$, and $\kappa_2$,
a ball $\sB$,
 and a sequence $\epsilon_n\rightarrow 0$, as $n\rightarrow\infty$, we have
\begin{equation}\label{PC2.2A}
\begin{aligned}
\widehat\cL^n \widetilde{\sV}(\Hat{x}) +
\widehat\cL^n g^n[\widetilde{\sV}](\Hat{x},k) 
&\,=\, \sA^n \widetilde{\sV}(\Hat{x}) + \sum_{i=1}^6 \cR^n_i[\widetilde{\sV}](\Hat{x},k)
\\
&\,\le\, \kappa_1\Ind_{{\sB}}(\Hat{x}) - \kappa_2 \widetilde{\sV}(\Hat{x}) + C_1 + 
\epsilon_n\widetilde{\sV}(\Hat{x})
\qquad \forall\,(\Hat{x},k)\in\widehat{\fX}^n\times\cK\,,
\end{aligned}
\end{equation}
Indeed the equality in \cref{PC2.2A} follows by \cref{L5.1}.
Following the calculation in the proof of \cref{L5.4}, and using \cref{EC2.2B},
the inequality in \cref{PC2.2A} follows by \cref{A2.3,L5.2}.
By \cref{A2.2} and \cref{EC2.2B}, we have
\begin{equation}\label{PC2.2B}
C_2\widetilde{\sV}(\Hat{x}) - C_3 
\,\le\, \widetilde{\sV}(\Hat{x}) + g^n[\widetilde{\sV}](\Hat{x},k)
\,\le\, C_3(\widetilde{\sV}(\Hat{x}) + 1) 
\qquad \forall\,(\Hat{x},k)\in\widehat{\fX}^n\times\cK\,,
\end{equation}
for some positive constants $C_2$ and $C_3$.
Combining \cref{PC2.2A,PC2.2B}, we see that
$V(\Hat{x},k)\df \widetilde{\sV}(\Hat{x}) + g^n[\widetilde{\sV}](\Hat{x},k)$
satisfies
$\widehat\cL^n V(\Hat{x},k) \le \kappa_3 \Ind_{\sB'}(x) - 
\kappa_4 V(\Hat{x},k)$ for
some positive constants $\kappa_3$ and $\kappa_4$, and a ball $\sB'$.
This together with \cref{PC2.2B}
and the hypothesis in \cref{EC2.2A} implies \cref{ET2.2A}, 
and completes the proof.
\end{proof}

\appendix

\section{The diffusion limit}\label{AppA}
\Cref{prop2.1} which follows, shows that under suitable assumptions,
the processes $\widehat{X}^n$ in \cref{E-hatX} and $\widehat{Y}^n$ in \cref{E-sde} 
have the same diffusion limit. 
This proposition is interesting in its own right.

Let $(\DD^d,\cJ_1)$ denote the space of $\Rd$-valued c\'{a}dl\'ag functions endowed with
the $\cJ_1$ topology (see, e.g., \cite{Patrick-99}).

\begin{proposition}\label{prop2.1}
Grant \cref{A2.2}. In addition, suppose that
$\widehat{X}^n(0) \Rightarrow y_0$, 
\begin{equation}\label{EP2.1A}
\frac{\xi^n_z(x^n_* + n^{\beta}\Hat{x},k) - \xi^n_z(x^n_*,k)}{n^{\beta}}
\,\xrightarrow[n\to\infty]{}\, \widehat{\xi}_z(\Hat{x},k) \qquad \forall\,
(k,z)\in\cK\times\sZ^n\,,
\end{equation}
uniformly on compact sets in $\RR^d$, $\widehat{M}^n$ is a square integrable
martingale, and
\begin{equation}\label{EP2.1B}
\frac{\Xi^n(x^n_*,k)}{n}\,\xrightarrow[n\to\infty]{}\, \overline{\Xi}(k)\,\in\,\RR^d
\qquad \forall\,k\in\cK\,.
\end{equation}
Then, $\widehat{X}^n$ and $\widehat{Y}^n$ have the same diffusion limit
$\widehat{X}$ in $(\DD^d,\cJ_1)$,
and $\widehat{X}$ is the strong solution of the SDE
\begin{equation*}
\D \widehat{X}(t) \,=\, \Bar{b}\bigl(\widehat{X}(t)\bigr)\,\D{t}
+ \sigma_\alpha \D W(t)\,,
\end{equation*}
with $\widehat{X}(0) = y_0$, where
$$
\Bar{b}(\Hat{x}) 
\,\df\, \sum_{k\in\cK}\pi_k\sum_{z}z\,\widehat{\xi}_z(\Hat{x},k)\,,$$ 
$$
(\sigma_\alpha)\transp\sigma_\alpha \,\df\, \begin{cases}
\sum_{k\in\cK}\pi_k\Bar{\Gamma}(k)\,, &\quad \text{for } \alpha>1\,, \\
\sum_{k\in\cK}\pi_k\Bar{\Gamma}(k) + \Theta\,, &\quad \text{for } \alpha = 1\,, \\
\Theta\,, &\quad \text{for } \alpha <1\,,
\end{cases}
$$
and $\Theta = [\theta_{ij}]$ 
is defined by $$\theta_{ij} \,\df\,
2\sum_{k,\ell\in\cK}\overline{\Xi}_i(k)\overline{\Xi}_j(\ell)\pi_k\Upsilon_{k\ell}\,,
\quad i,j\in\cI\,.$$
\end{proposition}

\begin{proof}
Recall that $\sum_{k\in\cK} \pi_k\Xi^n({x}^n_*,k) = 0$
and $\widehat{\Xi}^n(0,k) = n^{-\beta}\Xi^n({x}^n_*,k)$.
Recall the representation of $\widehat{X}^n$ in \cref{E-hatX}.
By \cite[Lemma 5.8]{PTW07},  $\widehat{M}^n$ is stochastically bounded;
see also the proof of \cite[Theorem 2.1 (i)]{ADPZ19}.
Since $\widehat{\Xi}^n$ is Lipschitz continuous by \cref{EA2.2B}, 
it follows by the same argument in the proof \cite[Lemma 5.5]{PTW07}
that $\widehat{X}^n$ is stochastically bounded.
Thus, applying \cite[Lemma 5.9]{PTW07}, $n^{-1}X^n$
converges to the zero process in $(\DD^d,\cJ_1)$.
We write $\widehat{X}^n$ as
\begin{equation}\label{ER2.2B}
\begin{aligned}
\widehat{X}^n(t) &\,=\, \widehat{X}^n(0) + \sum_{k\in\cK}\int_0^{t}
\bigl(\widehat{\Xi}^n(\widehat{X}^n(s),k) - \widehat{\Xi}^n(0,k)\bigr)
\Ind_k\bigl(J^n(s)\bigr)\,\D{s} + \widehat{M}^n(t)\\
&\mspace{150mu}
+ \sum_{k\in\cK}\frac{\Xi^n({x}^n_*,k)}{n}
n^{1 - \beta}\int_0^{t}\bigl(\Ind_k\bigl(J^n(s)\bigr) - \pi_k\bigr)\,\D{s}\,.
\end{aligned}
\end{equation}
Let $\Hat{S}^n(t)$ and $\Hat{R}^n(t)$ be $d$-dimensional processes  
denoting the second and fourth terms on the right-hand side of \cref{ER2.2B}.
It follows by \cite[Proposition 3.2]{ABMT} and \cref{EA2.2D} that
\begin{equation}\label{ER2.2C}
\Hat{R}^n \,\Rightarrow\, 
\begin{cases} W_R\,, &\quad \text{for } \alpha\,\le\,1\,, \\
0\,, &\quad \text{for } \alpha\,>\,1\,,
\end{cases}
\quad \text{in} \quad (\DD^d,\cJ_1)\,,
\end{equation}
as $n\rightarrow\infty$, where $W_R$ is a $d$-dimensional Wiener process
with the covariance matrix $\Theta$.
On the other hand, we have 
\begin{equation}\label{ER2.2D}
\begin{aligned}
\Hat{S}^n(t) \,=\, \sum_{k\in\cK}\int_0^{t}n^{-\nicefrac{\alpha}{2}}
\bigl(\widehat{\Xi}^n(\widehat{X}^n(s),k) - \widehat{\Xi}^n(0,k)\bigr)
\,\D\,\biggl(n^{\nicefrac{\alpha}{2}}\int_0^{s}\bigl(\Ind_k(J^n(u))
- \pi_k\bigr)\,\D{u}\biggr) \\
+ \sum_{k\in\cK}\pi_k\int_0^{t}\bigl(\widehat{\Xi}^n(\widehat{X}^n(s),k)
- \widehat{\Xi}^n(0,k)\bigr)\,\D{s}\,.
\end{aligned}
\end{equation}

It follows by the convergence of $n^{-1}X^n$ to the zero process that
$n^{-\nicefrac{\alpha}{2}}\widehat{X}^n$ also converges to the 
zero process uniformly on compact sets in probability.
Note that, for some constant $C$, we have
$\abs{\widehat{\Xi}^n(\widehat{X}^n(s),k)
- \widehat{\Xi}^n(0,k)} \le C\abs{\widehat{X}^n(s)}$ 
for all $s\ge 0$ by \cref{EA2.2B}.
It then follows by \cite[Proposition 3.2]{ABMT} 
and \cite[Theorem 5.2]{JMTW17} that the first term
on the right-hand side of \cref{ER2.2D} converges 
to the zero process uniformly on compact sets in probability, 
as $n\rightarrow\infty$.
See also the proofs of Lemma 4.4 in \cite{JMTW17} and Lemma 4.1 in \cite{ADPZ19}.
It is clear by \cref{EP2.1A} that
\begin{equation}\label{ER2.2E}
h^n(\Hat{x}) \,\df\, \sum_{k\in\cK}\pi_k\bigl(\widehat{\Xi}^n(\Hat{x},k)
- \widehat{\Xi}^n(0,k)\bigr) \,\longrightarrow\, 
\sum_{k\in\cK}\pi_k\sum_{z}z\,\widehat{\xi}_z(\Hat{x},k)
\end{equation}
uniformly on compact sets in $\RR^d$.
Note that the function $h^n$ is Lipschitz continuous by \cref{EA2.2B}.
By \cite[Theorem 4.1]{PTW07} (see also \cite[Lemma 4.1]{JMTW17}),
the integral mapping $x^n = \Psi^n(z^n)\colon \DD^d\rightarrow\DD^d$ defined by
$$
x^n(t) \,=\, z^n(t) + \int_0^{t} h^n(x^n(s)) \,\D{s}  
\qquad \forall\,n\in\NN\,,
$$
is continuous in $(\DD^d,\cJ_1)$.
Thus, applying the continuous mapping theorem and
using \cref{ER2.2B,ER2.2C,ER2.2D,ER2.2E}, we obtain
$$
\widehat{X}^n \,\Rightarrow\, \widehat{X} \quad \text{in} \quad (\DD^d,\cJ_1)\,.
$$
Recall the definitions
of $\Bar{\Gamma}^n$ and $\Theta^n$ in \cref{E-hatXi} and \cref{ES2.2B}, respectively. 
As $n\rightarrow\infty$, we have that
 $\Bar{\Gamma}^n(0,k)\rightarrow\Bar{\Gamma}(k)$ 
when $\alpha\ge 1$, and $\Bar{\Gamma}^n(0,k) \rightarrow 0$ when $\alpha < 1$
by \cref{EA2.2D}.
Since $\beta = \max\{1-\nicefrac{\alpha}{2},\nicefrac{1}{2}\}$, 
it then follows by \cref{EP2.1B} that $\Theta^n \rightarrow \Theta$ when $\alpha \le 1$, 
and $\Theta^n \rightarrow 0$ when $\alpha > 1$.
It is then straightforward to verify that
$\widehat{Y}^n\Rightarrow\widehat{X}$ in $(\DD^d,\cJ_1)$, 
as $n\rightarrow\infty$.
Therefore, $\widehat{X}^n$ and $\widehat{Y}^n$ have the same diffusion limit.
\end{proof}

\section*{Acknowledgments}
This research was supported in part by 
the Army Research Office through grant W911NF-17-1-001, and
in part by the National Science Foundation through grants
CMMI-1635410, DMS-1715210 and DMS-1715875.


\end{document}